\documentclass[12pt,a4paper]{amsart}
\pdfoutput=1

\usepackage{latexsym,amsmath,amssymb,amscd}
\usepackage[all]{xy}
\usepackage{graphicx}
\usepackage{appendix}
\usepackage{eucal}

\newtheorem{theorem}{Theorem}[section]
\newtheorem{corollary}[theorem]{Corollary}
\newtheorem{lemma}[theorem]{Lemma}
\newtheorem{proposition}[theorem]{Proposition}

\theoremstyle{definition}
\newtheorem{definition}[theorem]{Definition}

\theoremstyle{remark}
\newtheorem{remark}[theorem]{Remark}
\newtheorem{example}[theorem]{Example}

\theoremstyle{plain}

\newtheorem*{maintheorem}{The Extension Theorem}

\newcommand{\calb}{\mathfrak{B}}
\newcommand{\cale}{\mathcal{E}}
\newcommand{\calo}{\mathcal{O}}
\newcommand{\calt}{\mathcal{T}}

\newcommand{\norm}[1]{\| #1 \| }

\newcommand{\C}{\mathbb{C}}
\newcommand{\N}{\mathbb{N}}
\newcommand{\Z}{\mathbb{Z}}
\newcommand{\A}{\mathcal{A}}
\newcommand{\B}{\mathcal{B}}
\newcommand{\K}{\mathcal{K}}

\newcommand{\eb}{{\cale_B}}
\newcommand{\ebp}{{\cale_{B'}}}
\newcommand{\ebs}{{\cale_{B^*}}}
\newcommand{\ebc}{{\cale_{B^c}}}
\newcommand{\etb}{{\cale_{tB}}}
\newcommand{\ehgam}{{{}_H\cale_\Gamma}}
\newcommand{\esgam}{{{}_S\cale_\Gamma}}
\newcommand{\eggam}{{{}_G\cale_\Gamma}}
\newcommand{\estimesb}{\crho(S)\otimesh \eb}
\newcommand{\estimesbs}{\crho(S)\otimesh \ebs}
\newcommand{\egtimesb}{\crho(G)\otimesh \ebp}
\newcommand{\egtimesbbp}{\crho(G)\otimesh(\eb\oplus\ebp)}
\newcommand{\ehnn}{(\crho(G)\otimesh\eb)\oplus\eggam}

\newcommand{\dispestimesb}{\crho(S)\dispotimesh \eb}
\newcommand{\dispegtimesbbp}{\crho(G)\dispotimesh(\eb\oplus\ebp)}
\newcommand{\dispehnn}{(\crho(G)\dispotimesh\eb)\oplus\eggam}

\newcommand{\otimesh}{{\,\strut_{\scriptscriptstyle H}\hspace{-0.6ex}\mathop{\otimes}\,\,}}

\newcommand{\otimesg}{{\,\strut_{\scriptscriptstyle G}\hspace{-0.6ex}\mathop{\otimes}\,\,}}
\newcommand{\otimess}{{\,\strut_{\scriptscriptstyle S}\hspace{-0.6ex}\mathop{\otimes}\,\,}}
\let\dispotimesh=\otimesh
\let\dispotimesk=\otimesk
\let\dispotimesg=\otimesg
\newcommand{\olddispotimesh}{\!\mathop{\otimes}\limits_{\scriptscriptstyle\crho(H)}\!}
\newcommand{\dispotimesb}{\!\mathop{\otimes}\limits_{\scriptscriptstyle\crho(B)}\!}
\newcommand{\fre}{\mathfrak{E}}
\newcommand{\EH}[1]{\fre_{H#1}}
\newcommand{\Ep}{\fre^p}
\newcommand{\Epc}{\fre^{1-p}}
\newcommand{\EpH}[1]{\fre_{H#1}^p}
\newcommand{\EpcH}[1]{\fre_{H#1}^{1-p}}
\newcommand{\op}{\mathrm{op}}
\newcommand{\siggam}{\sigma^\Gamma}
\newcommand{\sigb}{\sigma^B}
\newcommand{\tim}{\tau}
\newcommand{\wbar}{\overline{w}}
\newcommand{\wpbar}{\overline{w'}}
\newcommand{\syl}{a}
\newcommand{\gen}{\overline{\syl}}
\newcommand{\tb}{{\calt_B}}
\newcommand{\symdiff}{{\scriptstyle\triangle}}
\newcommand{\crho}{{C^*_\rho}}
\newcommand{\tree}{\mathfrak{T}}
\newcommand{\gammal}{\Gamma_L}
\newcommand{\gammar}{\Gamma_R}
\newcommand{\ph}{p_H}
\newcommand{\phb}{p_{\!Hb\,}}
\newcommand{\ip}[1]{\left\langle#1\right\rangle}

\usepackage{eucal}
\usepackage{euler}
\usepackage{times}

\author{Jacek Brodzki}
\email{J.Brodzki@soton.ac.uk}

\author{Graham A. Niblo}
\email{G.A.Niblo@soton.ac.uk}

\author{Nick Wright}
\email{N.J.Wright@soton.ac.uk}
\address{ Mathematical Sciences, University of Southampton, Southampton, SO17~1BJ, England}


\title{$K$-theory and exact sequences of partial translation algebras}

\begin{document}
\maketitle

\begin{abstract}
In an earlier paper, the authors introduced partial translation algebras as a generalisation of group $C^*$-algebras. Here we establish an extension of partial translation algebras, which may be viewed as an excision theorem in this context. We apply this general framework to compute the $K$-theory of partial translation algebras and group $C^*$-algebras in the context of almost invariant subspaces of discrete groups. This generalises the work of Cuntz, Lance, Pimsner and Voiculescu. In particular we provide a new perspective on Pimsner's calculation of the $K$-theory for a graph product of groups.
\end{abstract}

\section{Introduction}

Partial translation algebras were introduced in \cite{BNW:partial-translations} as a generalisation of group $C^*$-algebras. In \cite{BNW:Trees} we demonstrated that a number of classical $C^*$-algebra extensions, including the Toeplitz and Cuntz extensions, arise naturally in this context. In this paper we develop a general framework  for constructing exact sequences of partial translation algebras that encompasses these and other examples, including the celebrated Pimsner-Voiculescu exact sequence. We use this to give a number of new computations of the $K$-theory for group $C^*$-algebras and in particular we provide a new perspective on Pimsner's calculation of the $K$-theory for a graph product of groups \cite{Pimsner}.

A \emph{partial translation} of a discrete metric space $B$ is a bijection $t$ whose domain and codomain are subsets of $B$, and with $d(x,t(x))$ a bounded function on the domain of $t$. Throughout this paper the spaces we consider will be subspaces of discrete groups. Let $\Gamma$ be a countable discrete group equipped with a left-invariant metric. A subspace $B$ of $\Gamma$ is naturally endowed with a canonical collection of partial translations $\{t^B_g \mid g\in \Gamma\}$ where
$$t^B_g\colon B\cap Bg \to Bg^{-1}\cap B,\qquad t^B_g(x)=xg^{-1}.$$
These partial translations of $B$ define partial isometries of $\ell^2(B)$ as follows. Let $\{\delta_x \mid x\in B\}$ denote the standard basis of $\ell^2(B)$. We define $T^B_g:\ell^2(B)\to \ell^2(B)$ by $T^B_g\delta_x=\delta_{t^B_g(x)}$ if $x\in B\cap Bg$, and $T^B_g\delta_x=0$ otherwise. By analogy with the right regular representation of a group, we denote the representation of the partial translations by $\rho(t^B_g)=T^B_g$. 
The \emph{partial translation algebra of $B$} is the 
 $C^*$-algebra generated by these partial isometries and is denoted $\crho(B)$. We will refer to the dense subalgebra of $\crho(B)$ generated algebraically by $T^B_g$, $g\in \Gamma$, as the \emph{translation ring}.

Note that in the case where $B=\Gamma$, the operator $\rho(t^\Gamma_g)$ is  the unitary on $\ell^2(\Gamma)$ given by right multiplication by $g^{-1}$; abridging notation we denote this by $\rho(g)$. Hence in this case $\rho$ is the usual right regular representation of $\Gamma$, the translation ring is the group ring of $\Gamma$, and
the translation algebra is the (right) reduced group $C^*$-algebra $\crho(\Gamma)$.

The structure of the translation algebra $C^*_\rho(B)$ is delicate and is unfortunately not functorial even with regards to inclusions between subspaces. Despite this inconvenient fact, we can characterise when an inclusion does induce an extension of translation algebras in terms of the following (coarse) geometric condition. 

\begin{definition} Let $B\subset X\subset \Gamma$ with $K$ the left stabiliser of $X$. 
We say that $B$ is  \emph{relatively deep in $X$} if for every coset $Kx$ contained in $X$, there exist points in $B\cap Kx$ which are arbitrarily far from the complement of $B$ in $X$. 
\end{definition}

The reader may find it illustrative to consider the case $X=\Gamma$ when the condition reduces to the statement that the complement of $B$ is not coarsely dense in $\Gamma$. In this case we simply say that $B$ is \emph{deep} in $\Gamma$.

The ideal arising from the extension is in general difficult to compute, however, when $B$ is ``relatively almost invariant'', as defined below, we can carry out the computation by exploiting the fact that the translation algebra $C^*_\rho(B)$ carries a natural representation on a $C^*_\rho(H)$-Hilbert module $\cale_B$, where $H$ is the left stabiliser of $B$. This module encodes the symmetries of $B$ inherited from $\Gamma$.

\begin{definition}\label{relative almost invariance}Let $B\subset X\subset \Gamma$ with $H$ the left stabiliser of $B$.The subset $B$ is \emph{relatively $H$-almost invariant in $X$} if it satisfies both of the following conditions:

\begin{enumerate}\label{relative-H-almost-invariant}
\item For all $g\in \Gamma$, there exists a finite subset $F\subset \Gamma$ such that \[(Bg\setminus B)\cap X\subseteq HF.\]
\item $B$ is co-separable in $\Gamma$, i.e., there exists a finite subset $F'\subset \Gamma$ such that $F'\cap (B\symdiff gB)=\emptyset$ if and only if $g\in H$.
\end{enumerate}
\end{definition}

While the conditions may seem contrived at first sight, in the case when $X=\Gamma$ the statement that $B$ is a relatively deep, relatively $H$-almost invariant subset of $\Gamma$  reduces to the more familiar notion that $B$ is a proper $H$-almost invariant subset of $\Gamma$. In particular in this case the second condition in definition \ref{relative almost invariance} follows from the first (see lemma \ref{ai_coseparable}). Almost invariant subsets arise naturally in low dimensional topology and geometric group theory, were first exploited by Stallings in his celebrated ends theorem, and, following Sageev's duality theorem, have been extensively studied in the context of group actions on CAT(0) cube complexes, \cite{Stallings, Scott, Sageev}.

We can now state our main result.

\begin{maintheorem}
Let $B$ and $X$ be subspaces of $\Gamma$ such that  $B\subset X\subset \Gamma$. Let also $H$ be the left stabiliser of $B$, $K$ the left stabiliser of $X$ and assume that $H\leq K$. 

 If   $B$ is relatively deep in $X$ then  there exists a short exact sequence of $C^*$-algebras:
\[
\begin{CD}
0 @>>> I(P,B) @>>> \crho(B) @>>>\crho(X)
@>>>0,\\
\end{CD}
\]
where the map $\crho(B) \to \crho(X)$ extends the assignment $T^B_g\mapsto T^X_g$, for all $g\in \Gamma$. 

Moreover, if $B$ is relatively $H$-almost invariant in $X$ then the ideal $I(P,B)$ is the algebra of compact operators on a suitable Hilbert module $\eb$ over the algebra $\crho(H)$. In particular the ideal is Morita equivalent to $\crho(H)$.
\end{maintheorem}

The short exact sequence in the extension theorem extends joint work with Putwain which appeared in \cite{Putwain}.

The extension theorem may be viewed as an excision theorem and is a  unification and generalisation of the following well known results:

\begin{enumerate}
\item \emph{The Toeplitz extension}\label{Toeplitz}: Applying the above Theorem with $B=\N, X=\Gamma=\Z$ we obtain the Toeplitz extension
\[
0\to \K \to \calt \to C(S^1) \to 0
\]
using the Fourier isomorphism $\crho(\Z) \cong C(S^1)$. 
\item  \emph{The Pimsner-Voiculescu sequence}:\label{Pimsner-Voiculescu}
Let $\Gamma = X=F_n$, the free group on $n$ generators $s_1, \dots, s_n$, and let $B$ be the subspace consisting of all reduced words not beginning with $s^{-1}_1$. We obtain the generalised Toeplitz extension introduced by Pimnser and Voiculescu in \cite{PV}:
\[
0 \to \mathcal K\to  \crho(B) \to \crho(F_n) \to 0. 
\]
 \item \emph{The Cuntz extension}:\label{Cuntz}
Let $B$ be the subspace of $F_n$ which consists of positive words in the generators $s_1, \dots, s_n$ together with the identity. Let $X$ be the union of the translates $s_1^kB$,  for all $k\in \Z$. The left stabiliser of $X$ is the group $H = \langle s_1\rangle $ and the left stabiliser of $B$ is trivial. We obtain an extension
\[
0 \to \mathcal K\to \crho(B) \to \crho(X)\to 0
\]
The middle and the right hand terms in this sequence are identified in  \cite[Theorem 4]{BNW:Trees} as the Cuntz algebra $\cale_n$ and the Cuntz algebra $\calo_n$, respectively. This way we obtain the Cuntz extension introduced in his computation of the $K$-theory of the algebra $\calo_n$ \cite{Cuntz:Annals}. 

\item \emph{Lance's Extension}:\label{Lance}
Let $\Gamma=G\mathop{*}S$ be a free product of countable discrete groups $G,S$ and let $B$ consist of all those elements $w\in \Gamma$ which are represented by reduced words not starting with a non-trivial syllable from $S$.
We obtain the short exact sequence constructed by Lance in \cite{Lance}.
\end{enumerate}

Lance's work was extended in the context of certain amalgamated free products by Natsume \cite{Natsume} and certain HNN-extensions by Anderson and Paschke \cite{Anderson-Paschke}.

Applying the extension theorem to the case when $\Gamma$ acts on a tree $\tree$ with no global fixed point allows us to unify all of these results. We conclude the paper by giving an explicit computation of the $K$-theory of the associated translation algebra $\crho(B)$ to obtain an alternative proof of Pimsner's celebrated result, \cite{Pimsner}.

Following Sageev's characterisation of almost invariant subsets of a group \cite{Sageev} as objects dual to non-trivial group actions on CAT(0) cube complexes, we obtain a potentially rich source of new algebra extensions arising from geometric group theory and low dimensional topology.

This research was supported in part by EPSRC grants EP/F031947/1 and EP/J015806/1.

\section{Partial translation algebras for deep subsets}\label{exact sequences}

We begin by considering the following question. If $B$ is a subspace of a discrete group $\Gamma$ then under what conditions does the assignment $T^B_g \mapsto \rho(g)$ induce a $*$-homomorphism $\crho(B)\to \crho(\Gamma)$? More generally, if $X,B$ are subspaces of $\Gamma$ with $B\subseteq X$, then under what conditions does $T^B_g\mapsto T^X_g$ induce a $*$-homomorphism $\crho(B)\to \crho(X)$?

In the case of the inclusion of $B$ into $\Gamma$ there is an obvious necessary condition: a product $T^B_{g_1}T^B_{g_2}\dots T^B_{g_n}$ in $\crho(X)$ must map to $\rho(g_1g_2\dots g_n)$ in $\crho(\Gamma)$, hence in particular such products must be non-zero. We will show that the condition that the monoid of products $T^B_{g_1}T^B_{g_2}\dots T^B_{g_l}$ has no zero is equivalent to a geometric condition on the subspace $B$, and that this is also a sufficient condition for the map $\crho(B)\to \crho(\Gamma)$ to exist.

\begin{proposition}
\label{notcoarselydense}

Let $\Gamma$ be a discrete group and $B \subseteq \Gamma$ a subset of $\Gamma$. The following are equivalent: 
\begin{enumerate}
\item $B$ is a deep subspace of $\Gamma$; 
\item For all $r>0$ there exists $x\in B$ such that $B_\Gamma(x,r)\subset B$;
\item \label{operatormonoid} the monoid $\{T^B_{g_1}T^B_{g_2}\dots T^B_{g_k} \mid g_i\in \Gamma\}$ has no zero. 
\end{enumerate}
\end{proposition}

Before proving this proposition we introduce some terminology. Compositions of operators $T^B_{g_1}T^B_{g_2}\dots T^B_{g_k}$, as appearing in part \ref{operatormonoid} of Proposition \ref{notcoarselydense}, correspond to successive applications of the group elements $g_k^{-1},g_{k-1}^{-1}$, etc. To encode this algebraically we define the \emph{track} of a sequence of group elements $(g_1, g_2, \dots,g_k)$ to be the pair $(g,F)$ where $F$ is the set $F=\{e,g_k,g_{k-1}g_k,\dots, g_1g_2\dots g_k\}$ of compositions and $g=g_1g_2\dots g_k$. Intuitively the track records which points are visited along a path, but not the order in which they are visited. With the composition defined below, the tracks form a monoid: up to a change of conventions, this is the Birget-Rhodes expansion of the group, first defined in \cite{Birget-Rhodes}.

We note that the operator $T^B_{g_1}T^B_{g_2}\dots T^B_{g_k}$ is determined by the track of $(g_1, g_2, \dots,g_k)$, and in particular two operators with the same track are equal. Specifically, if $(g,F)$ is the track of $(g_1, g_2, \dots,g_k)$ then for $x\in B$ we have
$$T^B_{g_1}T^B_{g_2}\dots T^B_{g_k}\delta_x=\begin{cases} \delta_{xg^{-1}}& \text{ if } xh^{-1}\in B \text{ for all }h\in F\\0&\text{ otherwise.}\end{cases}$$
As the track determines the operator we can denote by $T^B_{g,F}$ the translation operator for any sequence $(g_1,\dots,g_k)$ with this track. In Appendix \ref{universal} we examine spaces for which translation operators $T^B_{g,F}$\,, $T^B_{g',F'}$ are equal if and only if the tracks $(g,F), (g',F')$ are the same.

The composition of tracks is defined by $(g,F)\cdot(g',F')=(gg',Fg'\cup F')$. It is easy to check that this is associative and $(e,\{e\})$ is a global identity, hence the tracks form a monoid associated to the group. The composition of tracks is defined in such a way that it corresponds to the composition of translation operators.

\begin{proof}[Proof of Proposition \ref{notcoarselydense}]
It is clear that $(1)$ and $(2)$ are  equivalent. 

$(2)\,\implies\,(3)$: Let $(g,F)$ be the track of $(g_1,\dots,g_k)$. As $F$ is finite it lies in $B_\Gamma(e,r)$ for some $r$. By assumption there exists $x\in B$ such that $B_\Gamma(x,r)\subset B$. For all $h\in F$ we have $|h|\leq r$ so $xh^{-1}\in B_\Gamma(x,r)\subseteq B$. It follows that $T^B_{g_1}T^B_{g_2}\dots T^B_{g_k}\delta_x\neq 0$ so $T^B_{g_1}T^B_{g_2}\dots T^B_{g_k}\neq 0$.

$(3)\,\implies\,(2)$: Let $(g_1,\dots,g_k)$ be a sequence with track $(e,B_\Gamma(e,r))$. By assumption $T^B_{g_1}T^B_{g_2}\dots T^B_{g_k}\neq 0$, so in particular there exists $x\in B$ such that $T^B_{g_1}T^B_{g_2}\dots T^B_{g_k}\delta_x\neq 0$. Hence $xh^{-1}\in B$ for all $h\in F=B_\Gamma(e,r)$, i.e.\ $B_\Gamma(x,r)\subseteq B$.
\end{proof}

We are now in a position to construct a $C^*$-algebra extension associated with a pair of subspaces $B\subset X$ of a discrete group $\Gamma$. We  first fix some notation. Let $B^c= X\setminus B$ denote the complement of $B$ in $X$, and let $P$ be the projection of $\ell^2(X)$ onto the subspace $\ell^2(B^c)$. Let $\A$ be the $C^*$-algebra of operators on $\ell^2(X)$ generated by $\crho(X)$ and  the projection $P$.

\medskip

Note that in addition to $\crho(X)$, the algebra $\A$ also contains $\crho(B)$: this is the subalgebra of $\A$ generated by the operators $T^B_g=(1-P)T^X_g(1-P)$, for $g\in \Gamma$. We denote by $I(P)$ (resp.\ $I(1-P)$) the ideal generated by $P$ (resp.\ $1-P$) in the algebra $\mathcal{A}$, and denote by $I(P,B)$ the intersection of $I(P)$ with $\crho(B)$.

The following theorem generalises Theorem 8.3 of Putwain's thesis \cite{Putwain}.

\begin{theorem}\label{BNWsequence}
Let $X$ and $B$ be subspaces of a countable discrete group $\Gamma$, with $B\subset X$. If $B$ is relatively deep in $X$ then there exists a commutative ladder of short exact sequences:
\[
\begin{CD}
0 @>>> I(P,B) @>>> \crho(B) @>>>\crho(X)
@>>>0\\
@. @VVV @VVV @|\\
0 @>>> I(P)\cap I(1-P) @>>> I(1-P) @>>>\crho(X)
@>>>0.
\end{CD}
\]
The map $\crho(B) \to \crho(X)$ extends the assignment $T^B_g\mapsto T^X_g$, for all $g\in \Gamma$. 
\end{theorem}
\begin{proof}

With $\mathcal{A}$ and $P$ as in the discussion above we will first show that
\[
\mathcal{A}=\crho(X)+I(P)=\crho(B)+I(P).
\]
By construction, $\mathcal{A}$ contains   $\crho(X)$ and $I(P)$, and indeed it is the smallest $C^*$-algebra 
containing these. Since the sum of an ideal and a $C^*$-subalgebra is automatically closed, the first equality follows. 

Both $\crho(B)$ and $I(P)$ are contained in $\A$, so it suffices to show that $\crho(X)\subseteq \crho(B)+I(P)$. It is enough to show that $\crho(B)+I(P)$ contains the generators of $\crho(X)$, viz.\ $T^X_g$ for all $g\in \Gamma$. We have
\[
T^X_g=(1-P)T^X_g(1-P)+(1-P)T^X_gP+PT^X_g,
\]
where $(1-P)T^X_g(1-P)\in \crho(B)$ and $(1-P)T^X_gP+PT^X_g\in I(P)$. This establishes the second equality.

By the second isomorphism theorem, we have
\[
\crho(X)/(\crho(X)\cap I(P))\cong \crho(B)/I(P,B).
\]
As $\crho(B)\subseteq I(1-P)$ we likewise have
\[
\crho(X)/(\crho(X)\cap I(P))\cong I(1-P)/(I(P)\cap I(1-P)).
\]

To obtain the required exact sequences it remains to show that the intersection $\crho(X)\cap I(P)$ is zero.

Let $T$ be an element of $\crho(X)\cap I(P)$. As $T\in \crho(X)$, it is $H$-invariant, so that we 
have $\lambda(h)T = T\lambda(h)$, for all $h\in H$, where $\lambda$ is the representation of $H$ on $\ell^2(X)$ by 
left multiplication.  In particular for each $x\in X$ and $h\in H$ we have $\lambda(h)T\delta_x=T\delta_{hx}$, so $\norm{T\delta_x}=\norm{T\delta_{hx}}$.

Since $T$ is in $I(P)$, for every $\epsilon > 0$ there exists $T'\in I(P)$, such that $\norm{T-T'} < \epsilon$ and the support of $T'$ lies in the $R$-neighbourhood of $B^c\times B^c$, for some $R>0$. 

For any $x\in X$, as $B$ is $H$-deep, there exists $h\in H$, such that $d(hx, B^c) > R$, hence $T'\delta_{hx}= 0$
so $\norm{T\delta_{x}} =\norm{T\delta_{hx}}< \epsilon$. As this holds for all $x$ and $\epsilon$ we conclude that $T=0$, and so $\crho(X)\cap I(P)=\{0\}$.

Hence
\[
\crho(B)/I(P,B)\cong \crho(B)/(I(P)\cap I(1-P))\cong \crho(X).
\]
yielding the commutative ladder.

To check that $T^B_g\mapsto T^X_g$, we observe that as elements of $\mathcal{A}$, $T^B_g$ and $T^X_g$ differ by an element of $I(P)$, hence they have the same image under the quotient map $\mathcal{A}\mapsto \mathcal{A}/I(P)$.
\end{proof}

\begin{remark}\label{intersection} It follows directly from the definitions that when $B\subseteq X\subseteq \Gamma$ we have  $\crho(B) \subseteq (1-P)\A (1-P)$ and $I(P,B)\subseteq(1-P)I(P)(1-P)$. It can be shown that when $X=\Gamma$ these both become equalities.\end{remark}

When we pass to a subspace of a group $\Gamma$ we are breaking certain relations in the group, for instance in the Toeplitz example we break the relation $-1+1=0$ as $T^\N_{-1}T^\N_1\neq T^\N_0$. To make this precise, we can view $\Gamma$ as the quotient of the monoid of tracks by the relations $(g,F)=(g,F')$ for all $F,F'$. In the Toeplitz example the track of $(-1,1)$ is $(0,\{0,1\})$ and we have removed the relation $(0,\{0,1\})=(0,\{0\})$, while keeping the relation $(0,\{0,-1\})=(0,\{0\})$ since $T^\N_{1}T^\N_{-1}= T^\N_0$. In Appendix \ref{universal} we show that we can break \emph{all} of the relations to obtain a universal object in the category of subspaces of $\Gamma$.

\section{The Hilbert module representation for symmetric subspaces}\label{Hilbert modules}

Let $B$ be a subspace of a group $\Gamma$ and let $H$ be a subgroup of the left stabiliser of $B$. In this section we will construct a $\crho(H)$-Hilbert module $\eb$ on which the partial translation algebra $\crho(B)$ is naturally represented. The symmetries of the space given by $H$ result in symmetries of the operators in $\crho(B)$ and this is encoded by the fact that these are adjointable operators on the Hilbert module over $\crho(H)$. As the elements of $H$ are symmetries of the left action, it is convenient to follow the non-standard convention of using \emph{left} Hilbert modules.

We begin with a general construction. Let $A$ be a $C^*$-algebra and $p$ a projection in $A$. Then $pAp$ is a $C^*$-algebra, indeed a subalgebra of $A$, and the space $pA$ has the structure of a left $pAp$ module with product $pap\cdot pb=papb$. Moreover $pA$ is equipped with a $pAp$-valued inner product defined by $\ip{pa,pb}=pab^*p$ and the corresponding norm is the restriction to $pA$ of the norm on $A$. Thus $pA$ is complete and is hence a (left) $pAp$-Hilbert module.

More generally suppose $A$ is a $C^*$-algebra of operators on a Hilbert space $H$ and $p$ is a projection in $\B(H)$. If $pAp$ is contained in $pA$ then the closure $\overline{pAp}$ is a $C^*$-algebra and the closure $\overline{pA}$ is a $\overline{pAp}$-Hilbert module.

We will apply this to the algebra $\crho(B)$. Let $b\in B$ and let $p=\phb$ denote the projection of $\ell^2(B)$ onto $\ell^2(Hb)$. For $g\in\Gamma$ let $\sigma_g=p T^B_{b^{-1}g}$. The operator $p\rho(b^{-1}g)$ on $\ell^2(\Gamma)$ takes $\ell^2(Hg)$ to $\ell^2(Hb)$, hence the truncation $\sigma_g$ vanishes whenever $g\notin B$. Now for $g\in B$ and $g'\in G$ we have $\sigma_g T^B_{g'}=\sigma_{gg'}$ when $gg'\in B$ and is otherwise zero. If follows that the span of $\sigma_g$ for $g\in B$ is dense in $p\crho(B)$.

\begin{proposition}
If $b\in B$ and $p=\phb$ then $p\crho(B)p$ is contained in $p\crho(B)$. Moreover the algebra $p\crho(B)p$ is closed and is isomorphic to $\crho(H)$.
\end{proposition}

\begin{proof}
For $g\in\Gamma$ consider $pT^B_gp$. As remarked above the operators $\sigma_g$ span a dense subspace of $p\crho(B)$ so it follows that the operators $pT^B_gp$ span a dense subspace of $p\crho(B)p$. Writing $g=b^{-1}kb$, if $k\in H$ then $T^B_g$ preserves $\ell^2(Hb)$, while for $k\notin H$ it takes $\ell^2(Hb)$ to the orthogonal subspace $\ell^2(Hk^{-1}b)$. In the former case $pT^B_gp=pT^B_g\in p\crho(B)$ while in the latter case $pT^B_gp=0\in p\crho(B)$. This gives the required containment.

Identifying $pT^B_gp$ for $g=b^{-1}hb$ with $\rho(h)\in\C[H]$ yields an isomorphism between the $p$-truncation of the translation ring of $B$ with the group-ring of $H$. It follows that $p\crho(B)p$ is isomorphic to a dense subalgebra of $\crho(H)$. To see that the image is $\crho(H)$ and hence $p\crho(B)p$ is closed we construct an inverse map from $\crho(H)$ to $p\crho(B)p$. This is obtained by inducing from $\crho(b^{-1}Hb)$ to $\crho(\Gamma)$, truncating to $\ell^2(B)$ to obtain an element of $\crho(B)$ and then further truncating by $p$ to obtain an element of $p\crho(B)p$.
\end{proof}

We observe that the inner product $\ip{\sigma_g,\sigma_k}$ is $\rho(gk^{-1})$ when $gk^{-1}\in H$ and is zero otherwise. In particular this does not depend on the choice of coset $Hb\subseteq B$, hence different choices yield isomorphic Hilbert modules. We now make the following definition.

\begin{definition}\label{notation for Hilbert modules}
For $B$ a subspace of $\Gamma$ and $H$ a subgroup of the left-stabiliser of $B$, let ${}_H\cale_B$ denote the $\crho(H)$-Hilbert module $\overline{p\crho(B)}$ constructed above. When $H$ equals the left-stabiliser of $B$ we abridge our notation to $\eb$.
\end{definition}

For $\cale$ a Hilbert module over a group $C^*$-algebra $\crho(H)$ and $\mathcal{F}$ a Hilbert module equipped with a representation of $\crho(H)$ by adjointable operators we will abridge the notation for the tensor product $\mathcal{F}\olddispotimesh\cale$ as $\mathcal{F}\otimesh\cale$, for self-evident typographical reasons.

\begin{remark}\label{compacts on eb}
The algebra of compact operators on $\eb$ can be identified with $\crho(H)\otimes \K(\ell^2(H\backslash B))$. To see this let $\tb$ denote a transversal for the left action of $H$ on $B$, so that we have a bijection $B\leftrightarrow H\times \tb$ given by $hb\leftrightarrow (h,b)$ for $h\in H,b\in \tb$. Then we can identify the dense subspace of $\eb$ spanned by $\sigma_g$ for $g\in B$ with the algebraic tensor product $\C[H]\odot c_c(\tb)$. This identifies the inner product on $\eb$ with the tensor product of the inner product $\ip{a_1,a_2}=a_1a_2^*$ on $\C[H]$ by the usual $\C$-valued inner product on $c_c(\tb)$. Hence completing we obtain an isomorphism of Hilbert modules $\eb\cong \crho(H)\otimes \ell^2(\tb)$. Using the natural identification of $\tb$ with the quotient $H\backslash B$ yields the required isomorphism $\K(\eb)\cong\crho(H)\otimes \K(\ell^2(H\backslash B))$.
\end{remark}

As we are working with \emph{left} Hilbert modules, the adjointable operators are naturally written on the right. Certainly $\crho(B)$ acts on $p\crho(B)$ by right multiplication, and this extends to $\eb=\overline{p\crho(B)}$ by continuity. This gives a representation as adjointable operators since for $\xi,\eta\in \eb$ and $z\in\crho(B)$ we have $\ip{\xi x,\eta}=\xi x \eta^*=\ip{\xi ,x^*\eta}$. Hence right multiplication provides a representation of $\crho(B)$ on $\eb$. Explicitly for any $g\in B$ and $k\in \Gamma$ we have $\sigma_{g}T^B_k=\sigma_{gk}$.

Now, choosing a basepoint $b\in B$, we use the fact that the basis vector $\delta_b$ in $\ell^2(B)$ is a cyclic vector for the representation of $\crho(B)$ (that is $\crho(B)\delta_b$ is dense in $\ell^2(B)$) to show that the representation of $\crho(B)$ on $\eb$ is faithful. Given $z\neq 0$ in $\crho(B)$ there exists some basis vector $\delta_{bg}$, with $bg\in B$, such that $z^*\delta_{bg}\neq 0$. Writing $\delta_{bg}$ as $(T^B_g)^*\delta_b$ we have $z^*(T^B_g)^*p\neq 0$, since $\delta_b$ is in the range of $p$, and so $pT^B_gz$ is non-zero. Hence we see that for non-zero $z\in \crho(B)$, the image of $z$ in $\B(\eb)$ is non-zero, i.e.\ the representation of $\crho(B)$ on $\eb$ is faithful.

\begin{remark}
Taking the subgroup to be trivial, so that $\crho(H)=\C$, we obtain a faithful representation of $\crho(B)$ on the  left-$\C$-Hilbert module ${}_{\{e\}}\eb$. As a vector space this is simply the original Hilbert space $\ell^2(B)$, however as the operators are written on the right, the representation is a $*$-homomorphism $\crho(B)\to\B({}_{\{e\}}\eb)=\B(\ell^2(B))^\op$. The construction works for any subset $B$ of a group $\Gamma$: applying it when $B=\Gamma$ gives an opposite representation for the group $C^*$-algebra, $\crho(\Gamma)\to \B(\ell^2(\Gamma))^\op$. Returning to the general case where $B$ is a subset of $\Gamma$ and $H$ a subgroup of its left-stabiliser, we can form $\ell^2(H)\otimesh{}_H\eb$, where the representation of $\crho(H)$ on $\ell^2(H)$ is the opposite representation $\crho(H)\to \B(\ell^2(H))^\op$. The tensor product $\ell^2(H)\otimesh{}_H\eb$ is isomorphic to $\ell^2(B)$ and tensoring the representation of $\crho(B)$ on ${}_H\eb$  with $1$ on $\ell^2(H)$ yields the opposite representation of $\crho(B)$ on $\ell^2(B)$. In this way we can pass back from operators on ${}_H\eb$ to operators on Hilbert space, albeit with the unusual convention of operators acting on the right.
\end{remark}

\medskip

In the next section we will show, under suitable (almost invariance) hypotheses, that the ideal in the exact sequence of Theorem \ref{BNWsequence} is precisely the algebra $\K(\eb)$ of compact operators on the Hilbert module $\eb$. The ideal is therefore strongly Morita-equivalent to $\crho(H)$, indeed isomorphic to $\crho(H)\otimes \K(\ell^2(\tb))$ by Remark \ref{compacts on eb}.

\section{The ideal $I(P,B)$ for almost invariant subspaces}\label{almost-invariance}

The Toeplitz extension, the Pimsner-Voiculescu sequence and Lance's extension all arise from consideration of subspaces with ``small'' boundaries. In this section we will make the notion of ``small'' precise and in this context we will compute the ideal term $I(P,B)$ thus completing the proof of the extension theorem.

We start by recalling the classical definition of an almost invariant subset of a group. Let $H$ be a subgroup of $\Gamma$. We say that a subspace of $\Gamma$ is \emph{$H$-finite} if it is contained in a finite union of right $H$-cosets, $Hg_1\cup\dots\cup Hg_n$. A subspace $B$ of $\Gamma$ is \emph{$H$-almost invariant} if the symmetric difference $B{\scriptstyle\triangle} Bg$ is $H$-finite for all $g\in \Gamma$. It is well known that for a finitely generated group this is equivalent to the assertion that, taking the quotient of the Cayley graph of $\Gamma$ by the subgroup $H$ the image of $B$ has finite coboundary, i.e., the set of edges in the quotient with precisely one end point in the image of $B$ is finite. Hence (classic) $H$-invariant, $H$-almost invariant subsets correspond precisely to subsets of the quotient $H\backslash\Gamma$ which have small coboundary.

The following lemma relates the classical notion of almost invariant subset to our relative version:

\begin{lemma}\label{ai_coseparable}
Let $B$ be a subset of a group $\Gamma$ with left stabiliser $H$. Then $B$ is $H$-almost invariant if and only if it is relatively $H$-almost invariant with $B\subset X=\Gamma$. 
\end{lemma}

\begin{proof}
It is obvious that $H$-almost invariance is  equivalent to part (1) of definition \ref{relative almost invariance} by writing $B\symdiff Bg= (Bg\setminus B) \cup (Bg^{-1}\setminus B)g$. It therefore suffices to prove that $H$-almost invariance implies condition (2) of definition \ref{relative almost invariance}.  

If $B=\Gamma$ or $B=\emptyset$ there is nothing to prove, otherwise we may choose $b\in B$ and $c$ in the complement of $B$. We  note that the set \[\mathcal B=\{gB\mid g\in \Gamma, b\in gB, c\not \in gB\}\] is finite. To see this suppose that $b\in gB, c\not\in gB$ so that $g^{-1}b\in B\symdiff Bc^{-1}b$ which by $H$almost invariance is of the form $HF$ for some finite set $F$. Thus $g\in bF^{-1}H$ so $gB=bf^{-1} B$ for one of finitely many $f$. 

For each $gB\in \mathcal B\setminus\{B\}$ we choose an element $d\in B\symdiff gB$ and let $F'$ denote the union of these finitely many elements together with $b,c$. It is clear that for each of the $gB$ concerned we have $F'\cap (B\symdiff gB) \not = \emptyset$. On the other hand if $gB\not \in \mathcal B$ then either $b$ or $c$ will lie in $F'\cap (B\symdiff gB)$ so if $F'\cap (B\symdiff gB) = \emptyset$ then $gB = B$ as required.
\end{proof}
\bigskip
Lemma \ref{ai_coseparable} shows that in the classical case of almost invariance co-separability is automatic. An example to show that this is not always the case in the relative setting is given in Appendix \ref{example_coseparability}. As we will later see, finiteness of the coboundary of $H\backslash B$ in  $H\backslash \Gamma$ is precisely what is required to ensure that the ideal $ I(P,B)$ of elements of the algebra  $C^*_rB$ which represent the trivial element on $\ell^2(\Gamma)$ are precisely those which are represented as compact operators on the corresponding Hilbert module. Our relativisation of almost invariance gives the same result when representing  $C^*_rB$ on the Hilbert module arising from the inclusions $B\subset X\subseteq \Gamma$.  Condition (1) of relative almost invariance, Definition \ref{relative-H-almost-invariant}, is required to ensure that all  of the elements in $I(P,B)$ are compact, while condition (2), co-separability, is required to ensure that every compact operator arises in this way. The two  conditions may be viewed geometrically as separation properties. Classical almost invariance of $B$ can be is characterised by the property that each pair of  points in $\Gamma$ is separated by only a finite  collection of translates $gB$; co-separability can be characterised by the property  that the subset $B$ may be distinguished from all of its translates by a fixed finite collection of points, and it allows us to isolate single cosets of $H$ in a controlled way. It is this which allows us to demonstrate that the rank one operators can be realised by elements of $I(P,B)$.

\begin{lemma}\label{H_isolation}
Let $B$ be a subset of a group $\Gamma$ with left stabiliser $H$ such that $B$ is co-separable. Then there exist non-empty finite subsets $F_1, F_2\subset \Gamma$ such that 
\[
H= \bigcap\limits_{g\in F_1} Bg \setminus \bigcup\limits_{g\in F_2} Bg.
\] 

\end{lemma}

\begin{proof}
By definition there is a finite subset $F'$ such that $(B\symdiff gB)\cap F' $ is empty if and only if $g\in H$.  If necessary we enlarge $F'$ so that the two sets  $E_1=F'\cap B, E_2=F'\setminus B$ are both non-empty, and set $F_1=E_1^{-1}, F_2=E_2^{-1}$. 

First suppose $h\in H$. Then if $g\in E_1\subseteq B$, then $hg\in B$, as $H$ is the left stabiliser of $B$, so that $h\in Bg^{-1}$. On the other 
hand, if $g\in E_2$, then $g\not\in B$, so $hg\not\in B$ and $h\not\in Bg^{-1}$. It follows that 
\[
H \subseteq \bigcap\limits_{g\in F_1} Bg \setminus \bigcup\limits_{g\in F_2} Bg.
\]
Conversely, suppose that $k\in Bg^{-1}$ for all $g\in E_1$ and $k\not\in Bg^{-1} $ for all $g\in E_2$. For $g\in E_1$, 
$k\in Bg^{-1}$ so  $g\in k^{-1}B$. As $g$ is also in $B$,  we have that $g\not\in B\symdiff k^{-1}B$. 

Now if $g\in E_2$, $k\not\in Bg^{-1}$, so $g\not\in k^{-1}B$. As $g\not\in B$, again we have $g\not\in B\symdiff k^{-1}B$. Since 
$F'=E_1\cup E_2$, it follows that $(B\symdiff k^{-1}B)\cap F'$ is empty, so $k\in H$. This implies the reverse inclusion. 
\end{proof}

\subsection{Group actions on trees}\label{sub:Bass-Serre}
One important construction of almost invariant subspaces is furnished by actions on trees. Let $\tree$ be a tree and suppose that $\Gamma$ acts on $\tree$ without edge inversions (note that edge inversions can be removed by barycentrically subdividing). Removing an edge $\epsilon$ from $\tree$ divides the tree into two half-spaces. Choosing a vertex $v\in\tree$ selects one of these half-spaces, the one containing $v$. We denote this half-space by $\calb=\calb_{\epsilon,v}$. The vertex also provides a map from $\Gamma$ to $\tree$ defined by $g\mapsto gv$; given a pair $\epsilon,v$ we define $B=B_{\epsilon,v}$ to be the preimage of $\calb_{\epsilon,v}$ under this map.

The stabiliser of $\epsilon$ is the stabiliser of $\calb$ which is certainly contained in the left stabiliser of $B$. In the case that the orbit $\Gamma v$ is not contained in $\calb$, so $B$ is a proper subset of $\Gamma$ we have equality. As usual we denote the stabiliser by $H$. To see that $B$ is $H$-almost invariant, suppose $k\in B\setminus Bg$. This means that $kv\in \calb$ while $kg^{-1}v$ is not. Thus $v,g^{-1}v$ are separated by $k^{-1}\epsilon$. There are only finitely many edges separating $v,g^{-1}v$, so $k^{-1}$ lies in a finite union of left-$H$-cosets $g_1H\cup\dots g_nH$. Hence $k$ lies in a finite union of right-$H$-cosets, $Hg_1^{-1}\cup\dots Hg_n^{-1}$, i.e.\ $B\setminus Bg$ is $H$-finite. Similarly for $Bg\setminus B$, hence $B$ is $H$-almost invariant.

The right stabiliser of $B$ clearly contains the stabiliser of $v$. In general it will be bigger, however in the case that there is a single orbit of edges we have equality: if $k$ does not fix $v$ then there is a translate $g\epsilon$ of the edge $\epsilon$ separating $v,kv$. It follows that one of $g^{-1}v,g^{-1}kv$ is in $\calb$ and the other is not, so $k$ is not in the right stabiliser of $B$.

\bigskip

The case in which there is a single orbit of edges is of particular interest to us. In this case there are either one or two orbits of vertices depending on whether the quotient of $\tree$ by $\Gamma$ is a circle or an interval. The group splits as an $HNN$-extension in the former case and as a free product with amalgamation in the latter.

We consider first the case that $\Gamma$ is a free product with amalgamation $\Gamma=G\mathop{*}_H S$. Then $\Gamma$ acts on the Bass-Serre tree $\tree$ which has vertex set $\{\gamma G \mid \gamma\in \Gamma\}\cup\{\gamma S \mid \gamma\in \Gamma\}$. The action of $\Gamma$ is simply by left multiplication on these cosets. Two vertices $\gamma_1G, \gamma_2S$ are joined by an edge if their intersection $\gamma_1 G\cap \gamma_2 S$ is non-empty, in which case the intersection is $\gamma H$ for some $\gamma\in \Gamma$. Hence vertices are joined by an edge precisely when they can be written as $\gamma G,\gamma S$ for some $\gamma\in \Gamma$. Let $\epsilon$ be the edge labelled by $H$ and let $v=G$. The stabilisers of these are $H,G$ respectively and hence we have $B=B_{H,G}$ an $H$-almost invariant subspace of $\Gamma$ with left stabiliser $H$ and right stabiliser $G$. Selecting instead the vertex $S$ we obtain a different almost invariant subspace $B'=B_{H,S}$, with left stabiliser $H$ and right stabiliser $S$. These two subspaces are almost disjoint: the complement of $B'$ is the set $B^*=B\setminus H$.

Now consider an $HNN$-extension $\Gamma=G\mathop{*}_H$. Recall that this means $\Gamma$ is the group generated by $G$ along with one extra generator $t$, and subject to the additional relations $tht^{-1}=\theta(h)$ for $h\in H$, where $\theta$ is a monomorphism from $H$ to $G$. The Bass-Serre tree for $\Gamma$ is the tree $\tree$ whose vertices are the left cosets $\gamma G$ for $\gamma \in \Gamma$; the action of $\Gamma$ is again by left-multiplication. The edges are labelled by cosets $\gamma H$ and are naturally directed: the edge labelled $\gamma H$ goes from $\gamma t^{-1}G$ to $\gamma G$. Note that for $h\in H$, $\gamma ht^{-1}=\gamma t^{-1}\theta(h)$ so $\gamma h t^{-1}G=\gamma t^{-1}G$, making the inital vertex of the edge well-defined. We take $\epsilon$ to be the edge labelled by $H$, and take $v=G$. The edge $\epsilon$ goes from $t^{-1}G$ to $G$ and its stabiliser is $H$, while the stabiliser of $v$ is $G$. Hence we obtain an $H$-almost invariant subspace $B=B_{H,G}$ of $\Gamma$ with left stabiliser $H$ and right stabiliser $G$.

For both the free product with amalgamation and the $HNN$-extension the almost invariant sets can be described in terms of reduced words. Each element of $\Gamma=G\mathop{*}_H S$ can be written as a word $w=\syl_1\dots \syl_n$ in the alphabet $G\cup S$. If two consecutive letters $\syl_i,\syl_{i+1}$ both lie in $G$ or both lie in $S$ then the word can be simplified by replacing $\syl_i\syl_{i+1}$ by a single letter $\syl$ (the product of $\syl_i,\syl_{i+1}$ in $G$ or $S$). A word $\syl_1\dots \syl_n$ is \emph{reduced} if no two consecutive syllables lies in the same subgroup $G$ or $S$. Note that the condition that $w$ is reduced means that a syllable can lie in $G\cap S=H$ only in the case that $n=1$. With this concept of reduced word we can give a simple description of $B$ as the set
$$B=\{\syl_1\syl_2\dots \syl_n \mid \syl_1\in G \text{ and $\syl_1\syl_2\dots \syl_n$ is reduced}\}.$$

In the case of the $HNN$-extension $\Gamma=G\mathop{*}_H$, every element of $\Gamma$ can be written as an alternating word of the form $g_0t^{i_1}g_1\dots t^{i_n}g_n$ where each $g_k$ is in $G$ and each $i_k$ is $\pm 1$. A word containing $tht^{-1}$ with $h\in H$ can be simplified by replacing this subword by $k=\theta(h)$. Similarly $t^{-1}kt$, $k\in \theta(H)$ can be replaced by $h=\theta^{-1}(k)$. A word is said to be \emph{reduced} if it does not contain any subword of the form $tht^{-1}$ with $h\in H$ or $t^{-1}kt$ with $k\in \theta(H)$. The set $B$ can now be described as those group elements which cannot be written as a reduced word beginning with $t^{-1}$.

The case where $B$ is constructed from an action on a tree in this way will be considered further in Sections \ref{free-products} and \ref{HNN}.

\subsection{Computation of the ideal}

For this section, we shall assume that $B$ is a relatively $H$-almost invariant subspace of $X\subseteq\Gamma$, where $H$ is the left stabiliser of $B$ and is contained in the left stabiliser $K$ of $X$.
Recall that $P$ is the projection of $\ell^2(\Gamma)$ onto $\ell^2(B^c)$, the $C^*$-algebra generated by $P$ and $\crho(\Gamma)$ is denoted $\A$, and $I(P,B)$ is the intersection of the ideal $I(P)$ in $\A$ generated by $P$ with the subalgebra $\crho(B)$.

\begin{theorem}\label{ai-ideal}
Let $B$ be a relatively $H$-almost invariant subspace of $X\subseteq\Gamma$, where $H$ is the left stabiliser of $B$ and is contained in the left stabiliser $K$ of $X$.
Then $I(P,B) =\K(\eb).$
\end{theorem}
\begin{proof}
For simplicity of exposition we assume that  $H\subseteq B$, and take  $e\in H$ as the base point.

By Remark \ref{intersection}, $\crho(B)$ is contained in the corner $(1-P)\A(1-P)$ and the ideal $I(P,B)$ is contained in the corner $(1-P)I(P)(1-P)$. First we will show that $(1-P)I(P)(1-P)$ is contained in $\K(\eb)$ which will imply that $I(P,B)$ is also contained in the algebra 
of compact operators $\K(\eb)$.

  We denote by $P^g$ the projection $(T^X_g)^* P(T^X_g)$ onto $\ell^2((X\setminus Bg)\cap Xg)$. Products of elements $T^X_g$ and $P$ 
  span a dense subalgebra of $I(P)$ and we claim that any such product $T$ satisfies 
  $
  T= TP^g
  $ for some $g\in \Gamma$.  Indeed, the operator $T$ has the form $\dots PT^X_{g_1}\dots T^X_{g_n}$ for some $g_1, \dots, g_n$ in $\Gamma$. 
  Now let $g=g_1\dots g_n$ then we have
  \[
  PT^X_{g_1}\dots T^X_{g_n}= PT^X_{g_1}\dots T^X_{g_n} (T^X_g)^*T^X_g= T^X_{g_1}\dots T^X_{g_n} (T^X_g)^*PT^X_g
  \]
  where the second equality follows as $T^X_{g_1}\dots T^X_{g_n} (T^X_g)^*$ is a projection commuting with $P$. This establishes the claim. 
  
  It follows that operators of the form $(1-P)TP^g(1-P)$, with $T\in I(P)$, span  a dense subalgebra of $(1-P)I(P)(1-P)$. Since these operators 
  are adjointable operators on $\eb$, to show that they are in $\K(\eb)$ it suffices to check that $P^g(1-P)$ is compact.

  The projection $P^g(1-P)$ is the projection onto $\ell^2((B\setminus Bg)\cap Xg)$. 
 By relative $H$-almost invariance of $B$ and $H$-invariance of $X$, $ (B\setminus Bg)\cap X = HF$ for some finite set $F$. 
  
For $b\in B$ the operator $\ph T^B_{b}$ takes $\ell^2(Hb)$ to $\ell^2(H)$. Hence $\ph T^B_{b}P^{g}(1-P)$ equals $\ph T^B_{b}$ when
$Hb\subseteq HF$, and is otherwise zero. Thus $\sigma_b P^{g}(1-P)=\sigma_b$, for $b\in HF$, and otherwise vanishes. 
Projections onto single cosets are given by rank-one operators on the Hilbert module. Specifically, the operator
$\xi \mapsto \ip{\xi,\sigma_b}\sigma_b$ 
projects onto the coset $Hb$. 
Hence we see that $P^{g}(1-P)$ is a finite sum of rank-one operators as required. 

For the converse inclusion, the key step is to show that $\ph \in I(P,B)$. As an operator on $\eb$, this is the rank one operator defined by 
 $\xi \mapsto \ip{\xi,\sigma_e}\sigma_e$. If $\eta,\zeta$ are the images in $\eb$ of elements $y,z\in \crho(B)$ then the rank-one operator $\xi \mapsto \ip{\xi,\eta}\zeta$ is the composition $y^*\ph z$, and hence must also be in the ideal. Thus having established that $\ph \in I(P,B)$ it will follow that every rank one operator, and hence all of $\K(\eb)$, is contained in $I(P,B)$.

By co-separability Lemma \ref{H_isolation} gives us a formula for $H$, which using the fact that $H\subset B$ we can write 
 \[H= \bigcap\limits_{g\in F_1}(B\cap Bg) \cap \bigcap_{g\in F_2}(B\setminus Bg). 
 \] 
 The projection onto $\ell^2(B\cap Bg)$ 
is given by $(T^B_g)^*T^B_g\in \crho (B) $ while the projection onto $\ell^2(B\setminus Bg)$ is given by $T^B_e -(T^B_g)^*T^B_g\in \crho (B) $. It follows that the projection $\ph$ onto 
$\ell^2(H)$ is an element of $\crho(B)$. 

Now take any $g\in \Gamma$ such that $g^{-1}\in X\setminus B$. Then $Hg^{-1}$ is a subset of $X\setminus B$, so $H\subset (Xg\setminus Bg)\cap X$. It follows that 
$\ph = \ph P^g$, so $\ph$ is also in $I(P)$. This completes the proof that $\ph$ is an element of $I(P,B)$. 
\end{proof}

Combining Theorem \ref{ai-ideal} and Remark \ref{compacts on eb} with Theorem \ref{BNWsequence} we obtain:

\begin{maintheorem}
Let $B$ and $X$ be subspaces of $\Gamma$ such that  $B\subset X\subset \Gamma$. Let also $H$ be the left stabiliser of $B$, $K$ the left stabiliser of $X$ and assume that $H\leq K$. 

 If $B$ is relatively deep in $X$ then  there exists a short exact sequence of $C^*$-algebras:
\[
\begin{CD}
0 @>>> I(P,B) @>>> \crho(B) @>>>\crho(X)
@>>>0,\\
\end{CD}
\]
where the map $\crho(B) \to \crho(X)$ extends the assignment $T^B_g\mapsto T^X_g$, for all $g\in \Gamma$. 

Moreover, if $B$ is relatively $H$-almost invariant in $X$ then the ideal $I(P,B)$ is the algebra of compact operators on the Hilbert module $\eb$. In particular the ideal is Morita equivalent to $\crho(H)$.
\end{maintheorem}

In the special case of a classical almost invariant set in $\Gamma$ we get the corollary:
\begin{corollary}
Let $B$ be an $H$-infinite, $H$-almost invariant set in a group $\Gamma$. Then we obtain a short exact sequence:

\[
0 \to \crho(H)\otimes \K(\ell^2(H\backslash B)) \to \crho(B) \to\crho(\Gamma)
\to0.
\]

\end{corollary}

\section{Representations of $\crho(B)$ for convex subspaces}\label{sec:representations}

In this section we consider the problem of constructing representations of a translation algebra $\crho(B)$. This will be applied in later sections. Given a set of generators for the group $\Gamma$, each generator gives an element of the translation algebra $\crho(B)$. The question we consider here is under what conditions can an assignment of an operator to each of these generators, be extended to a representation of $\crho(B)$. The issue is the extent to which the relations from $\Gamma$ are broken in $\crho(B)$.

Let $\Sigma$ be an alphabet equipped with an involution which we denote $\syl\mapsto\syl^{-1}$. Let $R$ be a set of words in the alphabet $\Sigma$ such that for each $\syl\in\Sigma$ we have the relations $\syl\syl^{-1}$ and $\syl^{-1}\syl\in R$. The involution extends to words in the usual way, and we suppose that $R$ is closed under this, and under cyclic permutations. We take $\Gamma$ to be the group defined by the presentation $\langle \Sigma |R\rangle$, and for a word $w$ in the alphabet $\Sigma$ we denote by $\wbar$ the corresponding element of $\Gamma$. A key example, which we will use in Section \ref{free-products} is the presentation of $G\mathop{*}_HS$ where $\Sigma=G\cup S$, and $R$ consists of words of length at most $3$ with syllables taken just from $G$ or just from $S$, and such that the product is the identity as an element of $G$ or $S$ respectively. (Words of length $2$ give the required $\syl\syl^{-1}$ relations, while words of length $3$ implement the products on $G,S$, e.g.\ $g_1 g_2(g_1\cdot g_2)^{-1}=1$, where $g_1,g_2\in G$ and $\cdot$ denotes the multiplication in $G$.)

Let $B$ be a subset of $\Gamma$ containing the identity $e$. We say that a word $w=\syl_1\dots \syl_n$ \emph{stays in $B$} if for $k=1\dots n$ we have $\overline{\syl_1\dots \syl_k}\in B$. Otherwise we say that $w$ \emph{crosses out of $B$}. Here we are thinking of words as giving paths in the Cayley graph, starting at $e$; we remark that moving the basepoint out of $H$ can change whether or not the path stays in $B$. We suppose that $B$ is connected as a subset of the Cayley graph, so that every element of $B$ can be represented by a word staying in $B$.

We say that words $w,w'$ \emph{differ by a relation} if $w'$ can be obtained from $w$ by replacing a subword $v$ of $w$ with a word $v'$ such that the concatenation $v(v')^{-1}$ is in $R$.

\begin{definition}
A subset $B$ of $\Gamma$ is \emph{convex} for the representation $\langle \Sigma |R\rangle$ if whenever $w,w'$ are words staying in $B$ with $\wbar=\wpbar$ there is a sequence of words $w=w_1,w_2,\dots,w_m=w'$ such that each $w_k$ stays in $B$ and is obtained from the preceding word by applying a single relation.
\end{definition}

We define the \emph{boundary of $B$} to be the set of $b\in B$ such that there exists a generator $\syl$ in $\Sigma$ with $b\gen\notin B$. It is easy to see that the boundary is left-$H$-invariant.

The following proposition is motivated in part by the example $\Gamma=G\mathop{*}_HS$ with subspace $B=B_{H,G}$ from Section \ref{sub:Bass-Serre}. With the above presentation a relation stays in $B$ if it is a word in elements of $G$, and crosses out of $B$ if it is a word in elements of $S$, not all in $H$. In the former case the relation holds in $\crho(B)$, that is $T^B_{g_1}T^B_{g_2}T^B_{g_3}=T^B_{g_1g_2g_3}=1$, while in the latter case the relation gives a projection $T^B_{s_1}T^B_{s_2}T^B_{s_3}=1-\ph$ where as usual $\ph$ denotes the projection onto $\crho(H)$ in $\eb$.

\begin{proposition}\label{representations}
Let $\Gamma=\langle \Sigma |R\rangle$, and let $B$ be a convex subset of $\Gamma$ containing $e$, and with the boundary equal to the left-stabiliser, which we denote $H$. Let $\fre$ be a $\crho(H)$-Hilbert module equipped with a direct sum decomposition indexed by $H\backslash \Gamma$, that is $\fre=\bigoplus_{H\backslash \Gamma}\EH\gamma$. Let $\alpha$ be a map from $\Sigma$ to $\B(\fre)$ such that $\alpha(\syl^{-1})=\alpha(\syl)^*$, and extend $\alpha$ multiplicatively to words in $\Sigma$. Suppose:
\begin{enumerate}
\item for all  $\syl\in \Sigma$ and $\gamma\in \Gamma$, the operator $\alpha(\syl)$ takes $\EH{\gamma}$ to $\EH{\gamma\gen}$\,;

\item if $w$ in $R$ is a relation staying in $B$ then $\alpha(w)=1$;

\item there is a projection $p\in \B(\fre)$ with range in $\EH{}$ such that for every $w$ in $R$ crossing out of $B$ we have $\alpha(w)=1-p$;

\item \label{last hypothesis} the subspace $\EH{}$ carries a representation $\hat\rho$ of $\crho(H)$ commuting with $p$, and for each $h\in H$ there exists a word $w$ staying in $B$ with $\wbar=h$ such that $\hat\rho(h)$ is the restriction of $\alpha(w)$ to $\EH{}$.
\end{enumerate}
Then $\alpha$ extends to a representation of $\crho(B)$ on $\fre$, taking $T^B_{\gen}$ to $\alpha(\syl)$ for all $\syl\in \Sigma$.
\end{proposition}

\begin{proof}
The strategy of the proof is express $\fre$ as the sum of two $\crho(H)$-Hilbert submodules $\Ep,\Epc$, containing $\EH{}p,\EH{}(1-p)$ respectively, with this decomposition invariant under the operators $\alpha(w)$. We then show that $\Epc$ carries a representation of $\Gamma$ and identify $\Epc$ with $\EH{}(1-p)\otimesh \ehgam$ on which we have the representation $x\mapsto 1\otimes\pi(x)$ of $\crho(B)$, and we identify $\Ep$ with $\EH{}p\otimesh \eb$ on which we have the representation $x\mapsto 1\otimes x$.

Let $b\in B$ and let $w$ be a word staying in $B$ with $\wbar\in Hb$. We set $\EpH{b}=\EH{}p\alpha(w)$ and $\EpcH{b}=\EH{}(1-p)\alpha(w)$. We begin by showing that this is well-defined, i.e.\ it does not depend on the choice of representative $w$.

Let $w$ be a word staying in $B$, let $b=\wbar$, and let $v$ be a relation such that $wv$ stays in $B$. For $\xi\in\EH{}$ we have $\xi\alpha(w)$ in $\EH{b}$, hence if $b\notin H$, so $Hb\neq H$ then $\xi\alpha(w)(1-p)=\xi\alpha(w)$. Thus if $b\notin H$ then $\xi\alpha(w)\alpha(v)=\xi\alpha(w)$ whether or not $v$ stays in $B$. On the other hand if $b=\wbar\in H$ then as $wv$ stays in $B$ we see that $v$ also stays in $B$, so $\alpha(v)=1$. Hence we have shown that for $v$ a relation such that $wv$ stays in $B$ we have $\xi\alpha(wv)=\xi\alpha(w)$.

Applying this to relations of the form $\syl\syl^{-1}$ for $\syl\in \Sigma$, if $w$ is a word staying in $B$ then we can cancel such pairs in $w$ without changing $\xi\alpha(w)$. In particular, as the word $ww^{-1}$ cancels entirely we have $\xi\alpha(w)\alpha(w)^*=\xi\alpha(ww^{-1})=\xi$, for all $w$ staying in $B$.

Let $w,w'$ be words which differ by a relation, that is $w=w_1w_2w_3, w'=w_1w_2'w_3$ where $w_2(w_2')^{-1}\in R$. If $w,w'$ stay in $B$ then the concatenation $w_1w_2(w_2')^{-1}w_2'w_3$ also stays in $B$, so we have
$$\xi\alpha(w_1w_2w_3)=\xi\alpha(w_1w_2(w_2')^{-1}w_2'w_3)=\xi\alpha(w_1w_2'w_3).$$
By convexity of $B$ it follows that for $\xi\in \EH{}$ we have $\xi\alpha(w)=\xi\alpha(w')$ for any two words staying in $B$ with $\wbar=\wpbar$.

Now suppose that $w$, $w'$ are two words staying in $B$ with $H\wbar=H\wpbar$. Let $h=\overline{w'w^{-1}}$. By hypothesis \ref{last hypothesis} there is a word $v$ staying in $B$ such that $\hat\rho(h)$ equals $\alpha(v)$ on $\EH{}$. As $w',vw$ stay in $B$ and $\wpbar=\overline{vw}$ we have $\alpha(w')=\alpha(vw)=\hat\rho(h)\alpha(w)$ on $\EH{}$. The representation $\hat\rho$ commutes with $p$, hence $\EH{}p\alpha(w)=\EH{}p\hat\rho(h)\alpha(w)=\EH{}p\alpha(w')$ and similarly for $1-p$. Thus we see that $\EpH{b}$ and $\EpcH{b}$ are well-defined.

We have seen that if $w$ stays in $B$ then the restriction of $\alpha(w)$ to $\EH{}$ is an isometry, so $\EpH{b}$ is orthogonal to $\EpcH{b}$. Now let $\zeta\in \EH{b}$, and take $w=\syl_1\dots\syl_n$ to be a minimal word such that $b\wbar\in H$. By minimality, for each $k=1\dots n$ we have $b\overline{\syl_1\dots\syl_{k-1}}\notin H$, so $\zeta\alpha(\syl_1\dots\syl_{k-1}v)=\zeta\alpha(\syl_1\dots\syl_{k-1})$ for any relation $v$. Thus $\zeta\alpha(w)\alpha(w^{-1})=\zeta$. We have $\zeta\alpha(w)\in \EH{b\wbar}=\EH{}$ and $\overline{w^{-1}}\in Hb$ with $w^{-1}$ staying in $B$ (by minimality of $w$) hence we see that $\zeta$ is in the sum of $\EpH{b}$ and $\EpcH{b}$. Thus for all $b\in B$, $\EH{b}$ is the orthogonal direct sum of $\EpH{b}$ and $\EpcH{b}$.

Define $\Ep$ to be the sum of $\EpH{b}$ for $b\in B$, and define $\Epc$ to be the sum of those $\EpcH{b}$ for $b\in B$ together with those $\EH{b}$ for $b\notin B$. As $\EH{b}=\EpH{b}\oplus \EpcH{b}$ for each $b\in B$ we have $\fre=\Ep\oplus\Epc$. We now check that $\Ep,\Epc$ are invariant subspaces. First let $\xi\alpha(w)$ be an element of $\Ep$, where $\xi\in \EH{}p$ and $w$ stays in $B$. Let $\syl\in \Sigma$. If $w\syl$ stays in $B$ then certainly $\xi\alpha(w\syl)\in \Ep$. On the other hand if $w\syl$ does not stay in $B$ then $\wbar\in H$, $\xi\alpha(w)\in \EH{}p$, and $\syl\syl^{-1}$ is a relation which does not stay in $B$. Thus $\alpha(a\syl^{-1})=1-p$ so $\xi\alpha(w\syl\syl^{-1})=0$. As $\alpha(\syl)\alpha(\syl)^*$ is a projection $\alpha(\syl)$ is a partial isometry and so $\alpha(\syl\syl^{-1}\syl)=\alpha(\syl)\alpha(\syl)^*\alpha(\syl)=\alpha(\syl)$. Hence $\xi\alpha(w\syl)=\xi\alpha(w\syl\syl^{-1}\syl)=0\in \Ep$. Thus $\alpha(\syl)$ preserves $\Ep$ for all $\syl\in \Sigma$, and so more generally $\alpha(v)$ preserves $\Ep$ for any word $v$.

Now consider $\xi\alpha(w)\in\Epc$, where $\xi\in \EH{}(1-p)$ and $w$ stays in $B$, and let $\syl\in\Sigma$. Again if $w\syl$ stays in $B$ then $\xi\alpha(w\syl)\in \EH{}(1-p)$. If $w\syl$ does not stay in $B$ then $\overline{w\syl}\notin B$, so $\xi\alpha(w\syl)\in \EH{\overline{w\syl}}\subset \Epc$. It remains to check that if $\zeta\in \EH{\gamma}$ for $\gamma\notin B$ then $\zeta\alpha(\syl)\in \Epc$. If $\gamma\gen\notin B$ then this is immediate, so we suppose $\gamma\gen\in B$. This means we must in fact have $\gamma\gen\in H$ while $\gamma$ is not in $B$ and hence $\gen^{-1}$ is also not in $B$. So $\syl^{-1}\syl$ is a relation not staying in $B$ and therefore $\alpha(\syl^{-1}\syl)=1-p$. As before $\alpha(\syl\syl^{-1}\syl)=\alpha(\syl)$ so $\zeta\alpha(\syl)=\zeta\alpha(\syl)(1-p)\in\Epc$ as required.

As the subspace $\Epc$ is invariant it carries a unitary representation of $\Gamma$ defined by $\zeta\cdot \gamma=\zeta\alpha(w)$ where $w$ is any word with $\wbar=\gamma$: this is well-defined as for every relation $v\in R$, the operator $\alpha(v)$ is the identity on $\Epc$. For $\gamma\notin B$ define $\EpcH{\gamma}=\EH{\gamma}$. Then for each $\gamma\in \Gamma$, right multiplication by $\gamma$ induces an isomorphism from $\EpcH{}$ to $\EpcH{\gamma}$. This yields an isomorphism $\EpcH{}\otimesh \ehgam\cong\Epc$ defined by $\xi\otimes \sigma_\gamma\mapsto \xi\cdot\gamma$. It is easy to check that this preserves inner products since the representation is unitary and the subspaces $\EpcH{\gamma}$ are pairwise orthogonal. For $\syl\in \Sigma$, the isomorphism identifies $\alpha(\syl)$ with $1\otimes \rho(\gen)$ and hence the unitary representation of $\Gamma$ on $\Epc$ extends to a representation of $\crho(\Gamma)$. Composing with the quotient map $\pi:\crho(B)\to \crho(\Gamma)$ we obtain a representation of $\crho(B)$ on $\Epc$ extending the map $T^B_{\gen}\mapsto\alpha(\syl)$.

We now consider $\Ep$. We have an isomorphism $\EpH{}\otimesh \eb\cong \Ep$ defined by $\xi\otimes\sigma_b\mapsto \xi\alpha(w)$ where $w$ is a word staying in $B$ with $\wbar=b$: as we have seen above $\xi\alpha(w)$ does not depend on the choice of $w$. Taking $\xi,\xi'\in \EH{}$, words $w,w'$ staying in $B$, and $b=\wbar,b'=\wpbar$ we have
$\ip{\xi\alpha(w),\xi'\alpha(w')}=0$ unless $h=b(b')^{-1}\in H$, in which case $\xi\alpha(w)=\xi\hat\rho(h)\alpha(w')$ and
$$\ip{\xi\alpha(w),\xi'\alpha(w')}=\ip{\xi\hat\rho(h),\xi'\alpha(w')\alpha(w')^*}=\ip{\xi\hat\rho(h),\xi'}.$$
This agrees with the inner product $\ip{\xi\otimes\sigma_b,\xi'\otimes\sigma_{b'}}$ as required. Now let $\xi,\in \EH{}$, $w$ a word staying in $B$, and $b=\wbar$ and take $\syl\in \Sigma$. If $wa$ stays in $B$ then $\xi\alpha(wa)$ corresponds to $\xi\otimes \sigma_{b\gen}$ in the tensor product, whereas if $wa$ does not stay in $B$ then $\overline{wa}=b\gen$ is not in $B$ and we have $\xi\alpha(wa)=0$. Thus we see that the operator $\alpha(a)$ on $\Ep$ is identified with $1\otimes T^B_{\gen}$ on $\EpH{}\otimesh \eb\cong \Ep$. Using the isomorphism of Hilbert modules to transfer the representation $x\mapsto 1\otimes x$ of $\crho(B)$ on $\EpH{}\otimesh \eb$ to $\Ep$ yields a representation of $\crho(B)$ on $\Ep$ extending the map $T^B_{\gen}\mapsto\alpha(\syl)$. This completes the proof.
\end{proof}

\section{Free products with amalgamation} \label{free-products}

In this section we consider groups of the form $\Gamma=G\mathop{*}_HS$, where $G,S$ are discrete groups with common subgroup $H$. As noted in Section \ref{almost-invariance}, the splitting of the group gives rise to an $H$-infinite, $H$-almost invariant subset $B=B_{H,G}$ of $\Gamma$, and hence we obtain a short exact sequence
\[
0 \to \crho(H)\otimes \K \to \crho(B) \to\crho(\Gamma)
\to0.
\]
Recall that this is more canonically written as
\[
0 \to \K(\eb) \to \crho(B) \to\crho(\Gamma)
\to0
\]
where $\eb$ is the left $C^*_\rho(H)$-Hilbert module constructed in Section \ref{Hilbert modules}. The purpose of this section is to compute the $K$-theory of the algebra $\crho(B)$ in this case. Specifically we will show that
$$K_*(\crho(B))\cong K_*(\crho(G))\oplus K_*(\crho(S))$$
hence there is a 6-term exact sequence in $K$-theory computing $K_*(\crho(\Gamma))$ in terms of $K_*(\crho(H)),K_*(\crho(G))$ and $K_*(\crho(S))$. This is inspired by and  generalises Lance's computation in \cite{Lance}.

\bigskip

As the set $B$ is right-$G$-invariant there is a representation $\mu$ of $\crho(G)$ on $\ell^2(B)$. In fact the representation has range in $\crho(B)$, specifically, $\mu(\rho(g))=T^B_g$. Hence we can think of $\mu$ as giving a representation on $\eb$. Likewise, as noted in Section \ref{sub:Bass-Serre}, the subset $B^*=B\setminus H$ is right-$S$-invariant so there is a representation of $\crho(S)$ on $\ell^2(B^*)$. Extending by $0$ on $\ell^2(H)$ we view this as a non-unital representation $\nu$ of $\crho(S)$ on $\ell^2(B)$, defined by $\nu(\rho(s))=T^{B^*}_s\oplus 0_H$ for $s\in S$. Note that $\nu(\rho(e))$ is the projection onto $\ell^2(B^*)$, which equals $(T^B_t)^*T^B_t$ for any $t\in S\setminus H$ (but is not the same as $T^B_e$). Now for all  $s\in S$, 
$\nu(\rho(s))=T^{B^*}_s\oplus 0_H=T^{B}_s (T^B_t)^*T^B_t$, hence  the representation has range in $\crho(B)$, and again we can view this as a representation on $\eb$.

As in \cite{Lance} which considered the special case when $H=\{e\}$, the representations $\mu,\nu$ induce maps on $K$-theory and we will show that
$$\mu_*\oplus\nu_* : K_*(\crho(G))\oplus K_*(\crho(S))\to K_*(\crho(B))$$
is an isomorphism.

We remark that for $h\in H$ the images $\mu(h)$ and $\nu(h)$ are not equal: specifically we have $\nu(h)=(1-\ph)\mu(h)=\mu(h)(1-\ph)$. Since $\nu(e)=1-\ph$, for $h\in H$ and $s\in S$ we have $\mu(h)\nu(s)=\mu(h)(1-\ph)\nu(s)=\nu(h)\nu(s)=\nu(hs)$. Similarly $\nu(s)\mu(h)=\nu(sh)$.

The algebra $\crho(B)$ is generated by the images of the two representations $\mu,\nu$. To see this it suffices to check that for $w=\syl_1\dots\syl_n$ a reduced word we have the equality $T^B_w=T^B_{\syl_1}\dots T^B_{\syl_n}$. This follows from the fact that the boundary of the subset $\calb$ of $\tree$ is a single edge, hence a reduced word cannot cross out of $B$ and back in again. More precisely, given any $b\in B$, if for some $k$ we have $b\syl_1\dots \syl_k\notin B$ then $bw\notin B$. Hence if $\sigma_b T^B_{\syl_1}\dots T^B_{\syl_n}=0$ then $\sigma_b T^B_w=0$. On the other hand whenever $\sigma_b T^B_{\syl_1}\dots T^B_{\syl_n}$ is non-zero then it agrees with $\sigma_b T^B_w$ for any word $w$.

\subsection{Quasi-homomorphisms}
We will now move on to the construction of the maps $K_*(\crho(B))\to K_*(\crho(S))$ and $K_*(\crho(B))\to K_*(\crho(G))$, inverting $\mu_*,\nu_*$. In the spirit of \cite{Lance} this will be done by constructing quasi-homomorphisms, indeed the definition of the former map comes directly from \cite{Lance} (suitably adjusted to account for the amalgamation over $H$).

Let $J$ be an ideal in a $C^*$-algebra $A$. The \emph{double of $A$ over $A/J$}, denoted $D_{A,J}$, is the $C^*$-algebra of pairs $(a,b)$ which agree in the quotient $A/J$. That is, $D_{A,J}=\{(a,b)\in A\times A \mid a-b\in J\}$.

A \emph{quasi-homomorphism} from a $C^*$-algebra $B$ to a $C^*$-algebra $J$ is a pair of $*$-homomorphisms from $B$ to an algebra $A$ containing $J$ as an ideal, which agree in the quotient $A/J$. In other words, a quasi-homomorphism from a $B$ to $J$ is a $*$-homomorphism from $B$ to the double algebra $D_{A,J}$ for some $A\rhd J$. We denote the quasi-homomorphism by $(\phi,\phi'):B\rightrightarrows A\rhd J$.

The algebra $D_{A,J}$ contains $J\times\{0\}$ as an ideal, with quotient $A$. This extension is split by the map $a\mapsto (a,a)$, and hence is split at the level of $K$-theory. Thus identifying $K_*(J)$ with its image in $K_*(D_{A,J})$, a quasi-homomorphism $(\phi,\phi')$ from a $B$ to $J$ induces a map $K_*(B)\to K_*(J)$. This is given explicitly by the formula $[b]\mapsto [(\phi(b),\phi'(b))]-[(\phi'(b),\phi'(b))]$, and we denote the map on $K$-theory by $\phi_*-\phi'_*$.

The quasi-homomorphisms that we build will consist of maps to adjointable operators $\B(\cale)$ on various Hilbert modules $\cale$, with difference landing in the algebra of compacts $\K(\cale)$. The reader who is familiar with $KK$-theory will therefore recognise these as Kasparov triples.

We recall that the construction of $\eb$ applies for any subspace of a group and any subgroup of the left-stabiliser, and that our notation for the $\crho(K)$-Hilbert module of a subspace $X$ is ${}_K\cale_X$, or simply $\cale_X$ when $K$ is the whole of the left-stabiliser.

\begin{remark}
If $B\subseteq X$ are $K$-invariant subsets of $\Gamma$ then ${}_K\cale_B$ is a $\crho(K)$-Hilbert submodule of ${}_K\cale_X$, and moreover it is complemented, with complement ${}_K\cale_{X\setminus B}$.
\end{remark}

Our quasi-homomorphisms will be built using various bijections, the general case of which is described in the following remark.

\begin{remark}\label{right-equivariance}
Let $X$ be a subspace of a group $\Gamma$, with left-stabiliser $H$ and right-stabiliser $K$. For a subgroup $L$ of $\Gamma$ containing $H$, we denote by $L\times_H X$ the quotient of $L\times X$ by the diagonal action $h(k,x)=(kh^{-1},hx)$. Note that this space admits a left action by $L$ and a right action by $K$, where the left and right actions are on the $L$ and $X$ factors respectively. Suppose that the canonical map $L\times_H X\to LX\subseteq \Gamma$ is a bijection. Then this is left-$L$- and right-$K$-equivariant, and it induces a unitary from $\crho(L)\otimesh \cale_X$ to ${}_{L}\cale_{LX}$ taking $\rho(k)\otimes \sigma_x$ to $\sigma_{kx}$. As $X$ is right-$K$-invariant, $\cale_X$ has a unital representation of $\crho(K)$ by right multiplication, and tensoring with $1$ we obtain a representation of $\crho(K)$ on $\crho(L)\otimesh \cale_X$. Similarly ${}_{L}\cale_{LX}$ has a representation of $\crho(K)$ by right multiplication. The left-equivariance of the bijection ensures that the identification $\crho(L)\otimesh \cale_X\cong {}_{L}\cale_{LX}$ is a module map, while right-equivariance means that this conjugates the representations of $\crho(K)$ on these Hilbert modules.
\end{remark}

Applying the above remark we will now construct a quasi-homomorphism from the algebra $\crho(B)$ to $\K(\estimesb)$. Recall that $B$ is the preimage of a halfspace of $\tree$ under the map $\Gamma\to \tree, \gamma\mapsto\gamma G$, specifically we remove the edge $eH$ from $\tree$ and select the half-space $\calb$ containing the vertex $eG$. The action of $S$ on $\tree$ permutes transitively the edges adjacent to $eS$, and so $S\calb$ covers the complement of the vertex $eS$, in particular it covers all $G$-vertices. Hence $SB=\Gamma$. Moreover each $s\in S\setminus H$ takes $eH$ to some other edge and hence translates $B$ entirely off itself. Thus the canonical map $S\times_H B\to SB=\Gamma$ is a bijection. This is illustrated in Figure \ref{fig:bijection} in the case where $G=S=\Z$ and $H$ is trivial. The generators of $G,S$ correspond to horizontal and vertical edges respectively.

\begin{figure}
\mbox{\includegraphics[scale=0.7]{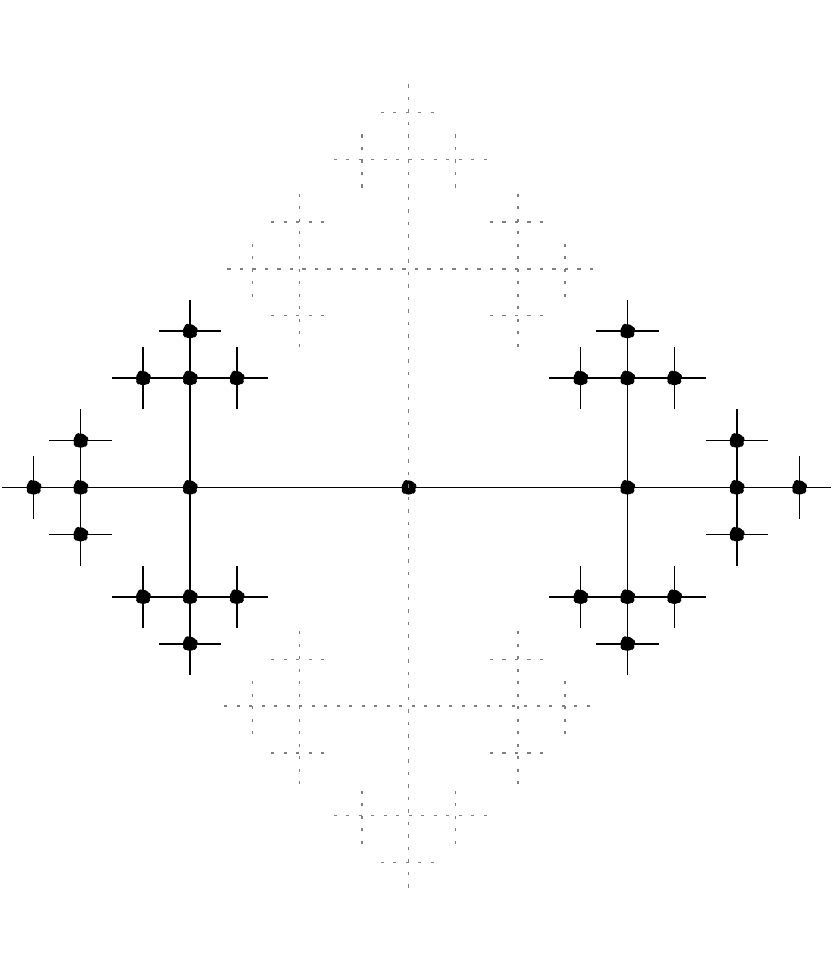}\quad \includegraphics[scale=0.7]{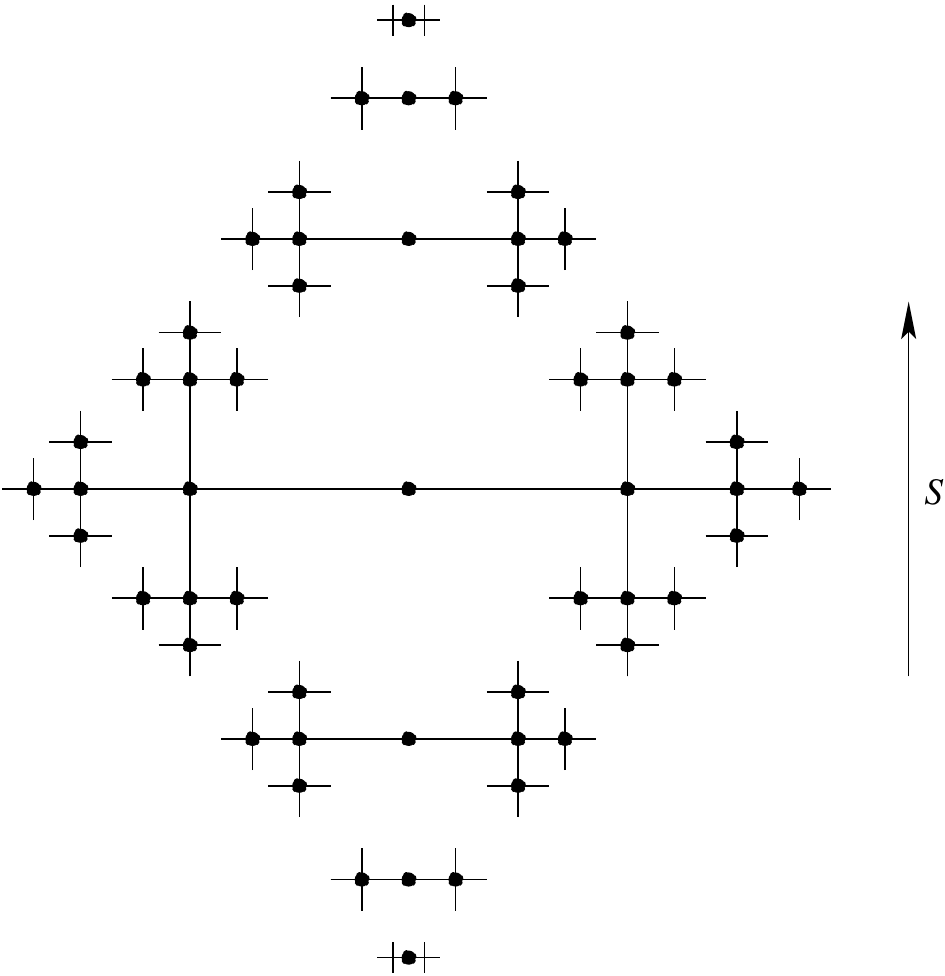}}

\caption{The decomposition of $\Gamma$ as $SB$ for $F_2=\Z\mathop{*}_{\{e\}}\Z$. The left hand figure illustrates $B\subseteq \Gamma$; the right-hand figure shows $\Gamma$ as the union of left-translates of $B$.}
\label{fig:bijection}
\end{figure}

We denote the induced unitary from $\esgam$ to $\estimesb$ by $U$. Note that as these are left Hilbert modules the map $U$ will be written on the right.

An element $x\in\crho(B)$ gives rise to an operator $1\otimes x$ on $\estimesb$, and, via the quotient map $\pi:\crho(B)\to \crho(G)$ of Theorem \ref{BNWsequence}, an operator $\pi(x)$ on $\esgam$. We define $\theta'(x),\theta(x)$ by
\begin{align*}
\theta'(x)&=1\otimes x\\
\theta(x)&=U^*\pi(x)U.
\end{align*}
Clearly $\theta,\theta'$ have range in adjointable operators on $\estimesb$ and we will show that that their difference lies in $\K(\estimesb)$, giving us a quasi-homomorphism.

\begin{remark}
In the case that $H$ is trivial $\eb$ can be identified with $\ell^2(B)$ (with the caveat that $\B(\eb)=\B(\ell^2(B))^\op$) and so $\estimesb=\crho(S)\otimes \ell^2(B)$. Tensoring over $\crho(S)$ with $\ell^2(S)$ we obtain operators $1\otimes \theta(x),1\otimes \theta'(x)$ on the Hilbert space $\ell^2(S\times B)$, and these recover the maps appearing in \cite{Lance}.
\end{remark}

\begin{lemma}\label{free-theta-quasi}
The pair $\theta',\theta$ define a quasi-homomorphism
$$\crho(B)\rightrightarrows \B(\dispestimesb)\rhd \K(\dispestimesb).$$
Moreover, the compositions on $K$-theory
\begin{align*}
K_*(\crho(G))\xrightarrow{\mu_*}&K_*(\crho(B))\xrightarrow{\theta_*-\theta'_*}K_*(\K(\dispestimesb))\\
K_*(\crho(S))\xrightarrow{\nu_*}&K_*(\crho(B))\xrightarrow{\theta_*-\theta'_*}K_*(\K(\dispestimesb))
\end{align*}
are respectively zero and the canonical isomorphism.
\end{lemma}

\begin{proof}
We begin by noting that $\theta(\mu(y))=\theta'(\mu(y))$ for $y\in \crho(G)$: by definition $\theta(\mu(y))=U^*\pi(\mu(y))U$ and $\pi(\mu(y))$ is simply $y$ action by right multiplication on $\esgam$. Since $\theta'(\mu(y))=1\otimes \mu(y)$, the equality follows from Remark \ref{right-equivariance}.

We will now show that for $z\in \crho(S)$, the difference $\theta(\nu(z))-\theta'(\nu(z))$ is the rank-one operator $\xi\mapsto \ip{\xi,1\otimes \sigma_e}(z\otimes \sigma_e)$. Since the images of $\mu,\nu$ generate $\crho(B)$ we will then conclude that the difference $\theta-\theta'$ always lands in $\K(\estimesb)$, hence $(\theta,\theta')$ defines a quasi-homomorphism. 

Now apply Remark \ref{right-equivariance} with $L=K=S$ and $X=B^*=B\setminus H$. We have a bijection $S\times_H B^*\leftrightarrow \Gamma\setminus S$, restricting the bijection $S\times_H B\leftrightarrow \Gamma$. The new bijection is right-$S$-equivariant and the corresponding unitary is a restriction of $U$ conjugating the right-multiplication representation of $\crho(S)$ to $1\otimes \nu$. Hence $\theta(\nu(z))=U^*zU$ preserves the subspace $\estimesbs$ and the restriction to this subspace coincides with the restriction of $\theta'(\nu(z))$.

Since the complement of $\ebs$ in $\eb$ is $\crho(H)$, the complement of $\estimesbs$ in $\estimesb$ is $\crho(S)\otimesh\crho(H)\cong \crho(S)$; this isomorphism is implemented by the restriction of $U$. Since $\nu(z)$ vanishes on $\crho(H)\subset\eb$ we see that $\theta'(\nu(z))$ is zero on this subspace, while the restriction of $\pi(\nu(z))$ to $\crho(S)$ is $z$. Using $U$ to identify the space with $\crho(S)\otimesh\crho(H)$ takes $z$ to the rank-one operator $\xi\mapsto \ip{\xi,1\otimes \sigma_e}(z\otimes \sigma_e)$. We thus see that as an operator on $\estimesb$ the difference $\theta(\nu(z))-\theta'(\nu(z))$ is simply this rank-one operator.

The observation that $\theta\circ\mu=\theta'\circ\mu$ ensures the vanishing of the composition $(\theta_*-\theta'_*)\circ\mu_*$ on $K$-theory. On the other hand the difference $\theta\circ\nu-\theta'\circ\nu$ takes $z\in\crho(S)$ to the rank-one operator $\xi\mapsto \ip{\xi,1\otimes \sigma_e}(z\otimes \sigma_e)$. This induces the canonical isomorphism $K_*(\crho(S))\to K_*(\K(\estimesb))$ as required.
\end{proof}

We now move on to the question of constructing a map from $K_*(\crho(B))$ to $K_*(\crho(G))$. Again we will do this by constructing a quasi-homomorphism. As remarked above, there is an asymmetry between the two groups $G,S$ created by the choice of basepoint $eG$ in the Bass-Serre tree yielding the right-$G$-invariant subspace $B$. For $B'$ the right $S$-invariant subspace obtained from choosing $eS$ intead we have a map $K_*(\crho(B'))\to K_*(\crho(G))$ which is directly analogous to the above map $K_*(\crho(B))\to K_*(\crho(S))$. The idea is to recreate $B'$ from $B$ and hence obtain a map $K_*(\crho(B))\to K_*(\crho(G))$.

The set $B'$ is `almost' the complement of $B$, specifically we have $\Gamma=B\cup B'$ and the intersection is $H$. Note that both $B,B'$ are left $H$-invariant and hence $B\sqcup B'$ admits a left action of $H$. Corresponding to this we take the direct sum of Hilbert modules $\eb\oplus \ebp$.

We have a bijection $G\times_HB'\leftrightarrow\Gamma$ given by $(g,b)\mapsto gb$. This is the direct analogue of the bijection $S\times_HB\leftrightarrow\Gamma$, and it is left-$G$- and right-$S$-equivariant. By Remark \ref{right-equivariance} we obtain a unitary isomorphism from $\eggam$ to $\egtimesb$, which we denote by $V_1$. Note that right multiplication by $\pi(x)\in\crho(\Gamma)$ defines an operator on $\eggam$. We use this to define $\psi':\crho(B)\to \egtimesbbp$ by
$$\psi'(x)=1\otimes x+V_1^*\pi(x)V_1.$$

We can also include $\Gamma$ into $B\sqcup B'$, in fact we can do this in two ways. Decomposing $\Gamma$ as $B\sqcup B^c$ gives one inclusion into $B\sqcup B'$, while decomposing $\Gamma$ as $(B')^c\sqcup B'$ gives a second inclusion (differing on $H$). Recall that $(B')^c=B^*=B\setminus H$. Note that the first decomposition of $\Gamma$ is into right-$G$-invariant sets, the latter into right $S$-invariant sets; the decompositions are both left-$H$-invariant. These inclusions yield two isometries from $\ehgam$ into $\eb\oplus\ebp$ which we denote by $V_2^G,V_2^S$ respectively. We now define two maps $\psi^G$, $\psi^S$ by
\begin{align*}
\psi^{G}(x)&=1\otimes (V_2^{G})^*\pi(x)V_2^{G}\\
\psi^{S}(x)&=1\otimes (V_2^{S})^*\pi(x)V_2^{S}.
\end{align*}
We remark that $\psi^G,\psi^S$ are very closely related, indeed $\psi^G(x)$ is obtained from $\psi^S(x)$ by conjugating by the operator $1\otimes v$ where $v$ is the unitary involution on $\eb\oplus\ebp$ which exchanges $\crho(H)\oplus0$ with $0\oplus\crho(H)$ while leaving $\ebs\oplus\ebc$ fixed.

\begin{lemma}\label{free-psi-quasi}
The pair $\psi',\psi^S$ define a quasi-homomorphism
$$\crho(B)\rightrightarrows \B(\dispegtimesbbp)\rhd \K(\dispegtimesbbp).$$
Moreover, the compositions on $K$-theory
\begin{align*}
K_*(\crho(G))\xrightarrow{\mu_*}&K_*(\crho(B))\xrightarrow{\psi'_*-\psi^S_*}K_*(\K(\dispegtimesbbp))\\
K_*(\crho(S))\xrightarrow{\nu_*}&K_*(\crho(B))\xrightarrow{\psi'_*-\psi^S_*}K_*(\K(\dispegtimesbbp))
\end{align*}
are respectively the canonical isomorphism and zero.
\end{lemma}

\begin{proof}
As in Lemma \ref{free-theta-quasi}, the cancellation between $\psi$ and $\psi'$ comes from applying Remark \ref{right-equivariance}, though here we need to consider restriction to the subspaces $\crho(G)\otimesh\eb$ and $\crho(G)\otimesh\ebp$ separately.

We begin by showing that $\psi^S(\nu(z))=\psi'(\nu(z))$ for $z\in\crho(S)$. By definition $\psi^S(\nu(z))=(V_2^{S})^*\pi(\nu(z))V_2^{S}$, and the operator $\pi(\nu(z))$ on $\eggam$ is given by right-multiplication by $z$. The restriction of $(V_2^{S})^*zV_2^{S}$ to $\eb$ is then given by restricting $z$ to $\ebs$ and then extending by $0$ to $\eb$. This is $\nu(z)$, so the restriction of $\psi^S(\nu(z))$ to $\crho(G)\otimesh\eb$ is $1\otimes \nu(z)$ which equals the restriction of $\psi'(\nu(z))$.

Next consider the restriction of $\psi'(\nu(z))$ to $\crho(G)\otimesh\ebp$. By Remark \ref{right-equivariance} applied to $L=G$, $X=B'$, $K=S$, the unitary $V_1$ conjugates right-multiplication by $z$ to $1\otimes \nu'(z)$ where $\nu'$ is the unital representation of $\crho(S)$ on the $S$-invariant subspace $B'$ of $\Gamma$. On the other hand $\psi^S(\nu(z))=1\otimes (V_2^{S})^*zV_2^{S}$ and the restriction of $(V_2^{S})^*zV_2^{S}$ to $\ebp$ is precisely $\nu'(z)$, hence we see that the restrictions of $\psi^S(\nu(z))$ and $\psi'(\nu(z))$ again agree. Thus $\psi^S(\nu(z))=\psi'(\nu(z))$ for $z\in\crho(S)$.

We now move on to $\mu(y)$ for $y\in \crho(G)$. First we will compute the difference $\psi'(\mu(y))-\psi^G(\mu(y))$, and then we will relate this to $\psi^S$. The restriction of $(V_2^{S})^*\pi(\nu(y))V_2^{S}=(V_2^{S})^*yV_2^{S}$ to $\eb$ is simply $y$, hence the restrictions of $\psi'(\mu(y)),\psi^G(\mu(y))$ to $\crho(G)\otimesh\eb$ are both $1\otimes y$. In particular they agree.

On $\crho(G)\otimesh\ebp$ we further restrict to $\crho(G)\otimesh\ebc$, and apply Remark \ref{right-equivariance} for $L=K=G$ and $X=B^c$. As usual this shows that the restrictions agree. It remains to consider the restrictions to the Hilbert submodule $\crho(G)\otimesh(0\oplus \crho(H))$ of $\egtimesbbp$. The range of $V_2^G$ misses $0\oplus \crho(H)$ hence $\psi^G(\mu(y))$ vanishes on this subspace. The subspace is isomorphic to $\crho(G)$, with the isomorphism implemented by the restriction of $V_1$, hence the action of $\psi'(\mu(y))$ on this subspace is given by $y$. Thus $\psi'(\mu(y))-\psi^G(\mu(y))$ is the rank-one operator
$$\xi\mapsto \ip{\xi,1\otimes (0\oplus\sigma_e)}(y\otimes (0\oplus\sigma_e)).$$

We have therefore shown that the image of $\psi'(\mu(y))-\psi^G(\mu(y))$ is compact for $y\in \crho(G)$. Conjugating $\psi^G(\mu(y))$ by $1\otimes v$ yields $\psi^S(\mu(y))$, hence as $v$ is a compact perturbation of the identity we deduce that $\psi^G(\mu(y))-\psi^S(\mu(y))$ and hence $\psi'(\mu(y))-\psi^S(\mu(y))$ are also compact. Thus $(\psi',\psi^S)$ defines the required quasi-homomorphism.

We have seen that $\psi',\psi^S$ agree on the image of $\nu$, yielding vanishing of the composition $(\psi'_*-\psi^S_*)\circ\nu_*$ on $K$-theory. The pair $(1,1\otimes v)$ lies in the double algebra for the compact operators inside the adjointable operators on the module $\egtimesbbp$. Hence $(\psi',\psi^G)$ is also a quasi-homomorphism, and defines the same map on $K$-theory as $(\psi',\psi^S)$. The computation of $(\psi'-\psi^G)\circ\mu$ shows that the composition $(\psi'_*-\psi^G_*)\circ\mu_*$ induces the canonical isomorphism on $K$-theory, hence the same holds replacing $\psi^G$ with $\psi^S$, which completes the proof.
\end{proof}

In view of the observation that $\psi'_*-\psi^S_*=\psi'_*-\psi^G_*$ we will drop the $S,G$ superscripts from our notation except when we need to be explicit in computations.

\subsection{The isomorphism on $K$-theory}

Combining Lemmas \ref{free-theta-quasi},\ref{free-psi-quasi} we deduce that the map $\mu_*\oplus\nu_*:K_*(\crho(G))\oplus K_*(\crho(S))\to K_*(\crho(B))$ is a split injection, with left inverse given by $(\psi'_*-\psi_*)\oplus(\theta_*-\theta'_*)$. To complete the computation of $K_*(\crho(B))$ it remains to prove that the quasi-homomorphisms also provide a right inverse. We begin by identifying the composition on $K$-theory with an explicit quasi-homomorphism.

We define $\phi,\phi':\crho(B)\to\B(\crho(B)\otimesh(\eb\oplus\ebp))$ as follows. Note that $\nu(1)$ is the projection in $\crho(B)$ corresponding to the subspace $B^*$. Viewing $\crho(B)$ as a Hilbert module over itself, $\crho(B)\nu(1)$ is a complemented submodule and is isomorphic to $\crho(B)\otimess\crho(S)$, where $\crho(S)$ acts on $\crho(B)$ via the representation $\nu$. Moreover, as $B^*$ is right-$H$-invariant, $\nu(1)$ commutes with the representation of $\crho(H)$ on $B$, and hence defines an adjointable operator $\nu(1)\otimes 1$ on $\crho(B)\otimesh \eb$, $(a\otimes \xi)(\nu(1)\otimes 1)=a\nu(1)\otimes \xi$.

We recall that $\theta(x)=U^*\pi(x)U$ is an operator on $\estimesb$, hence we can form the operator $1\otimes\theta(x)$ on $\crho(B)\otimess(\estimesb)\cong\crho(B)\nu(1)\otimesh\eb$. As the latter is a complemented subspace of $\crho(B)\otimesh\eb$ we can extend the operator to $\crho(B)\otimesh\eb$, defining it to be $p_0\otimes x$ on the complement $\crho(B)p_0\otimesh \eb$, where $p_0$ is the complementary projection $1-\nu(1)$.

Similarly we have the operator $V_1^*\pi(x)V_1$ on $\egtimesb$ (this is the restriction of $\psi'(x)$ to $\egtimesb$) and we form the operator $1\otimes V_1^*\pi(x)V_1$ on $\crho(B)\otimesg(\egtimesb)\cong \crho(B)\otimesh\ebp$. Finally we have $(V_2^S)^*\pi(x)V_2^S$ on $\eb\oplus\ebp$, and we form $1\otimes (V_2^S)^*\pi(x)V_2^S$ on $\crho(B)\otimesh(\eb\oplus\ebp)$. We now define
$$\begin{matrix}
\phi(x)&=& \Bigl(p_0\otimes x\,\,\oplus\,\,1\otimes(U_1^*\pi(x)U_1)\Bigr) \,\oplus\,1\!\otimes(V_1^*\pi(x)V_1), \\
\phi'(x)&=& 1\otimes (V_2^S)^*\pi(x)V_2^S.
\end{matrix}$$

We remark that for $h\in H$ we have $\phi(T^B_h)=1\otimes (T^B_h\oplus T^{B'}_h)$.
To see this first we note that $U_1^*\pi(T^B_h)U_1=\theta(T^B_h)$ and as $T^B_h$ is in the image of $\mu$ this agrees with $\theta'(T^B_h)=1\otimes T^B_h$, cf.\ Lemma \ref{free-theta-quasi}. Thus $1\otimes(U_1^*\pi(T^B_h)U_1)=\nu(1)\otimes T^B_h$. Similarly $V_1^*\pi(T^B_h)V_1$ is the restriction of $\psi'(T^B_h)$ to $\crho(G)\otimesh\ebp$. The operators $\psi'(T^B_h), \psi'(T^{B^*}_h)$ agree on this submodule, so this is $1\otimes \nu'(\rho(h))=1\otimes T^{B'}_h$, cf.\ Lemma \ref{free-psi-quasi}.

\begin{lemma}\label{the composition}
The pair $(\phi,\phi')$ defines a quasi-homomorphism,
$$\crho(B)\rightrightarrows \B(\crho(B)\otimesh(\eb\oplus\ebp))\rhd \K(\crho(B)\dispotimesh(\eb\oplus\ebp)).$$
The induced map on $K$-theory is the composition
\begin{equation}\label{free-composition}
K_*(\crho(B))\to K_*(\crho(G))\oplus K_*(\crho(S))\to K_*(\crho(B))
\end{equation}
where the first map is given by $(\psi'_*-\psi_*)\oplus(\theta_*-\theta'_*)$, and the second by $\mu_*\oplus \nu_*$.
\end{lemma}

In the statement of the composition (\ref{free-composition}), we are simplifying notation using the stability of $K$-theory: $K_*(\K(\egtimesbbp))$ is isomorphic to $K_*(\crho(G))$ by Morita equivalence, etc.

\begin{proof}[Proof of Lemma \ref{the composition}]
We begin by writing down the two compositions
$$K_*(\crho(B))\to K_*(\crho(G))\to K_*(\crho(B))$$
and
$$K_*(\crho(B))\to K_*(\crho(S))\to K_*(\crho(B))$$
as quasi-homomorphisms. The map $K_*(\crho(B))\to K_*(\crho(G))$ is given by the pair
$$\psi',\psi^S:\crho(B)\rightrightarrows \B(\crho(G)\otimesh(\eb\oplus\ebp))\rhd \K(\crho(G)\otimesh(\eb\oplus\ebp)).$$
Now forming $\crho(B)\otimesg(\crho(G)\otimesh(\eb\oplus\ebp))\cong \crho(B)\otimesh(\eb\oplus\ebp)$, the composition with $\mu_*$ is given by
$$1\otimes\psi',1\otimes\psi^S:\crho(B)\rightrightarrows \B(\crho(B)\otimesh(\eb\oplus\ebp))\rhd \K(\crho(B)\otimesh(\eb\oplus\ebp)).$$
Similarly for $\theta',\theta$ hence the composition (\ref{free-composition}) is given by the quasi-homomorphism
$$\eta=1\otimes\psi'\oplus 1\otimes \theta,\quad\eta'=1\otimes\psi^S\oplus 1\otimes \theta'$$
where $\eta,\eta'$ take values in the adjointable operators on the Hilbert module
$$\fre:=\Bigl(\crho(B)\dispotimesh\eb\,\oplus\,\crho(B)\dispotimesh\ebp\Bigr)\,\,\oplus\,\,\crho(B)\nu(1)\dispotimesh\eb.$$
The remainder of the proof involves simplifying this.

On $\crho(B)\otimesh \eb$, note that $\eta(x)$ is the restriction of $1\otimes \psi'(x)$ which is given by $1\otimes x$, identifying $\crho(B)\otimesg(\crho(G)\otimesh\eb)$ with $\crho(B)\otimesh\eb$. Likewise on $\crho(B)\nu(1)\otimesh\eb$, see see that $\eta'(x)$ is given by \mbox{$1\otimes \theta'(x)=1\otimes x$}, identifying the space $\crho(B)\nu(1)\otimess(\crho(S)\otimesh\eb)$ with $\crho(B)\nu(1)\otimesh\eb$.

Let $\fre_1$ denote the copy of $\crho(B)\nu(1)\otimesh\eb$ appearing as a submodule of the first summand $\crho(B)\otimesh\eb$ of $\fre$, and let $\fre_2$ denote the final summand $\crho(B)\nu(1)\otimesh\eb$. Let $w$ be the involution on $\fre$ which interchanges $\fre_1$ with $\fre_2$; this is an adjointable operator as $\fre_1,\fre_2$ are complemented in $\fre$. Let $\eta^w=w^*\eta w$. The standard homotopy $w_t=\frac 12 ((1+e^{i\pi t})+(1-e^{i\pi t})w),t\in[0,1]$ connects $w=w_1$ with the identity.

Since $\eta,\eta'$ agree modulo compact operators on $\fre_2$, while the restriction of $\eta$ to $\fre_1$ is equal to the restriction of $\eta'$ to $\fre_2$ (both are given by $x\mapsto 1\otimes x$) we see that the restrictions of $\eta$ to $\fre_1,\fre_2$ agree modulo compact operators. Thus $\eta$ commutes with $w$, and hence $w_t$, modulo compact operators. The homotopy $w_t^*\eta w_t$ is thus constant modulo compacts, so $(w_t^*\eta w_t,\eta')$ is a homotopy of maps into the double algebra. Hence the quasi-homomorphism $(\eta^w,\eta')$ induces the same map on $K$-theory as $(\eta,\eta')$.

The restrictions of $\eta^w,\eta'$ to $\fre_2$ now agree exactly. Since we have exact cancellation we can restrict $\eta^w,\eta'$ to $\crho(B)\otimesh(\eb\oplus\ebp)$, to obtain a quasi-homomorphism inducing the same map on $K$-theory. It is now a simple matter of comparing definitions to verify that these restrictions are precisely $\phi,\phi'$.
\end{proof}

To complete the proof that $K_*(\crho(B))$ is isomorphic to $K_*(\crho(G))\oplus K_*(\crho(S))$, it remains to show that the quasi-homomorphism $(\phi,\phi')$ induces the canonical isomorphism on $K$-theory. We prove this in the following Theorem.

The proof will make use of Proposition \ref{representations}, so before we embark on this we will verify that $B$ is convex for a suitable presentation of $\Gamma$ as required for this. We take the presentation of $\Gamma$ as described in Section \ref{sec:representations}. The process of reducing a word $w=\syl_1\dots\syl_n$ to a reduced word simply involves replacing pairs of letters $\syl_i\syl_{i+1}$ by a single letter $\syl$ which is their product in either $G$ or $S$. The word $w'$ so obtained differs by the single relation $\syl_i\syl_{i+1}\syl^{-1}$ in $R$. To say that a word stays in $B$ means that its track is a subset of $B$; as we are replacing pairs of letters by a single letter (representing the same element of the group) the track of $w'$ is a subset of that for $w$, and hence if $w$ stays in $B$ then so does $w'$.

Now suppose that $w=\syl_1\dots\syl_n,w'=\syl'_1\dots\syl'_n$ are two reduced words representing the same element of $B$. Then $w'$ can be obtained from $w$ by a sequence of moves of the form $\syl_i\syl_{i+1}\rightsquigarrow \syl_i''\syl_{i+1}''$ where for some $h\in H$, $\syl_i''$ is the product of $\syl_i$ with $h$, while $\syl_{i+1}''$ is the product of $h^{-1}$ with $\syl_{i+1}$. This is a two-step process: $\syl_i\syl_{i+1}\rightsquigarrow \syl_i h\syl_{i+1}''\rightsquigarrow \syl_i''\syl_{i+1}''$. The first step adds the element $\syl_1\dots\syl_i h$ to the track. This stays in $B$ as $\syl_1\dots\syl_i\in B$ and $B$ is right-$H$-invariant. The second step is another reduction, hence stays in $B$. By reducing and converting between reduce words in this way we can get between any two words representing the same element. Thus $B$ is convex.

Recall that the boundary of $B$ is defined to be the set of $b\in B$ such that $b\syl\notin B$ for some $\syl$ in the generating set $G\cup S$. Since $B$ is right-$G$-invariant we need only consider $\syl\in S$, and as $B^*=B\setminus H$ is right-$S$-invariant it follows that the boundary is a subset of $H$. As the boundary is non-empty and left-$H$-invariant it is $H$.

\begin{theorem}\label{free product theorem}
The map $\mu_*\oplus\nu_*:K_*(\crho(G))\oplus K_*(\crho(S))\to K_*(\crho(B))$ is an isomorphism, with inverse provided by $(\psi'_*-\psi_*)\oplus(\theta_*-\theta'_*)$.
\end{theorem}

\begin{proof}
We will prove that $(\phi,\phi')$ induces the canonical isomorphism on $K$-theory by a homotopy argument. To avoid ambiguity of notation, we write $\sigma_b$, $b\in B$ for elements of $\eb$, and write $\sigma'_b$, $b\in B'$ for elements of $\ebp$, in particular $\sigma_e$ denotes an element of $\eb$ not $\ebp$. We begin by observing that for $x\in\crho(B)$ the operator $\phi'(x)$ leaves the subspace $\crho(B)\otimesh (\ebs\oplus\ebp)$ invariant, and vanishes on the complementary subspace $\crho(B)\otimesh(\crho(H)\oplus 0)\cong \crho(B)$. We will show that $\phi(x)$ is homotopic to the sum of $\phi'(x)$ and the rank-one operator $\xi\mapsto \ip{\xi,1\otimes \sigma_e}(x\otimes \sigma_e)$. We denote this sum by $\phi''(x)$.

The homotopy $\Phi_\tim$ that we construct will be such that $(\Phi_\tim,\phi')$ is a quasi-homomorphism for all $\tim$. We will then be able to conclude that the quasi-homomorphism $(\phi,\phi')$ and $(\phi'',\phi')$ induce the same map on $K$-theory. The latter induces the canonical isomorphism as required.

We begin by comparing $\phi$ with $\phi''$. In Lemma \ref{free-theta-quasi} we showed that if $x=\nu(z)$, $z\in\crho(S)$ then $\theta(x)$ is the sum of $\theta'(x)=1\otimes x$ and the rank-one operator $\xi\mapsto \ip{\xi,1\otimes \sigma_e}(z\otimes \sigma_e)$. The restriction of $\phi(x)$ to $\crho(B)\otimesh\eb$ is $p\otimes x\oplus 1\otimes\theta(x)$, hence we see that this is $1\otimes x$ plus the rank-one operator $\xi\mapsto \ip{\xi,1\otimes \sigma_e}(\nu(z)\otimes \sigma_e)$. On the other hand the restriction of $\phi'(x)$ is simply $1\otimes x$ (cf.\ proof of Lemma \ref{free-psi-quasi}). Thus the restriction of $\phi''(x)$ agrees with $\phi(x)$, for $x=\nu(z)$.

Now we consider the restrictions of $\phi(x),\phi'(x),\phi''(x)$ to $\crho(B)\otimesh\ebp$, for $x=\nu(z)$, $z\in\crho(S)$. In Lemma \ref{free-psi-quasi} we showed that $\psi^S(x)$ agrees with $\psi'(x)$ on this subspace, hence tensoring with $1$, the restrictions of $\phi(x)$ and $\phi'(x)$ agree. On this subspace $\phi''(x)=\phi'(x)$ hence we deduce that $\phi''(x)=\phi(x)$ for $x=\nu(z)$.

Now consider $x=\mu(y)$, for $y\in\crho(G)$. In Lemma \ref{free-theta-quasi} we showed that for such $x$, $\theta(x)=\theta'(x)=1\otimes x$. Hence we see that the restriction of $\phi(x)$ to $\crho(B)\otimesh\eb$ is given by $1\otimes x$. This agrees with the restriction of $1\otimes \psi^G(x)$, which is obtained from $\phi'(x)$ by conjugating by $1\otimes v$. (Recall that $v$ is the unitary involution on $\eb\oplus\ebp$ which exchanges $\crho(H)\oplus0$ with $0\oplus\crho(H)$ while leaving $\ebs\oplus\ebc$ fixed.)

On the other hand, again by Lemma \ref{free-theta-quasi}, the restrictions of $\psi'(x)$ and $\psi^G(x)$ to $\crho(B)\otimesh\ebp$ differ by the rank-one operator $\xi\mapsto \ip{\xi,1\otimes \sigma'_e}(y\otimes \sigma'_e)$. Tensoring with $1$, we see that the restriction of $\phi(x)$ is the restriction of $1\otimes \psi^G(x)$ plus the rank-one operator $\xi\mapsto \ip{\xi,1\otimes \sigma'_e}(\mu(y)\otimes \sigma'_e)$. Hence we see that $\phi''(x)$ is not equal to $\phi(x)$ but rather to the conjugate of this by $1\otimes v$.

Motivated by this we take $\Phi_\tim$ to be constant on $T^B_s$ for $s\in S\setminus H$, setting this equal to $\phi(T^B_s)$, while for $g\in G$ we conjugate $\phi(T^{B}_g)$ by the element $1_B\otimes v_\tim$ where $v_\tim$ is the standard homotopy $\frac 12 ((1+e^{i\pi \tim})+(1-e^{i\pi \tim})v)$. Since $\Phi_0,\Phi_1$ agree with respectively $\phi,\phi''$ on the generators of $\crho(B)$, to construct the homotopy it remains to show that $\Phi_\tim$ extends to a representation of $\crho(B)$ for $t\in (0,1)$. As $\phi(T^B_h)=1\otimes (T^B_h\oplus T^{B'}_h)$ for $h\in H$ we observe that $v$ and hence $v_\tim$ commute with $\phi(T^{B}_h)$, so we have $\Phi_\tim(T^B_{h})=\phi(T^{B}_h)$.

We now work over the Hilbert module
$$\fre=\eb\dispotimesb \bigl(\crho(B)\dispotimesh(\eb\oplus\ebp)\bigr)\cong \eb\dispotimesh(\eb\oplus\ebp).$$
We will show that for each $\tim$, $1\otimes\Phi_\tim$ extends to a representation of $\crho(B)$ on $\fre$, using Proposition \ref{representations}. Following the notation of this proposition, we will write $\alpha(\syl)$ for $1\otimes\Phi_\tim(T^B_\syl)$, and we extend this to words in the usual way.

Let $\gamma\in \Gamma$, and define $\EH\gamma$ to be the closed submodule of $\fre$ generated by $\sigma_b\otimes \sigma_w$ where $b,w\in B$ with $bw\in H\gamma$, and by $\sigma_b\otimes \sigma'_w$ where $b\in B$, $w\in B'$ with $bw\in H\gamma$. We remark that each elementary tensor of this form lies in precisely one such submodule; it follows that the spaces $\EH\gamma$ are mutually orthogonal and $\fre=\bigoplus_{H\gamma\in H\backslash\Gamma}\EH\gamma$. We observe that $1\otimes\phi(T^B_g)$ takes $\EH\gamma$ to $\EH{\gamma g}$ for $g\in G$, and $1\otimes\phi(T^{B}_s)$ takes $\EH\gamma$ to $\EH{\gamma s}$ for $s\in S\setminus H$. Since $1\otimes v$ preserves each $\EH\gamma$, we deduce that $\alpha(\syl)$ takes $\EH\gamma$ to $\EH{\gamma\syl}$ for $\syl\in G\cup S$.

If $w$ is a relation which stays in $B$ then $w=gg^{-1}$ with $g\in G$ or $g_1g_2g_3$ where $g_1,g_2,g_3\in G$ and $g_1g_2g_3=e$ in $G$. Then as $\phi$ is a representation of $\crho(B)$ we have that $\alpha(gg^{-1})$ and $\alpha(g_1g_2g_3)$ are given by conjugating $1\otimes \phi(T^B_gT^B_{g^{-1}})$ and $1\otimes \phi(T^B_{g_1}T^B_{g_2}T^B_{g_3})$ respectively by $1\otimes v_\tim$. Since $T^B_gT^B_{g^{-1}}=1$ and $T^B_{g_1}T^B_{g_2}T^B_{g_3}=1$ we have that $\alpha(gg^{-1})$ and $\alpha(g_1g_2g_3)$ are both $1$.

Let $p$ denote the rank one projection of $\fre$ onto the submodule generated by $\sigma_e\otimes \sigma_e$. Let $w$ be a relation which does not stay in $B$. This means that $w$ is $ss^{-1}$ with $s\in S\setminus H$ or $w=s_1s_2s_3$ with $s_1,s_2,s_3\in S$, not all in $H$, and $s_1s_2s_3=e$ in $S$. We have $\alpha(ss^{-1})=1\otimes \phi(T^B_sT^B_{s^{-1}})$, while $\alpha(s_1s_2s_3)=1\otimes \phi(T^B_{s_1}T^B_{s_2}T^B_{s_3})$. The operators $T^B_sT^B_{s^{-1}}$, $T^B_{s_1}T^B_{s_2}T^B_{s_3}$ are both equal to $\nu(1)$ which is the projection of $\eb$ onto $\ebs$. Thus we have $\alpha(ss^{-1})=\alpha(s_1s_2s_3)=1-p$.

The restriction of $1\otimes\phi(T^B_h)$ to $\EH{}$ provides a representation of $\crho(H)$ which we denote $\hat\rho$, and as observed above this agrees with $\alpha(h)=1\otimes \Phi_\tim(T^B_h)$. The fact that the representation $\hat\rho$ commutes with $p$, follows from the fact that the representation of $\crho(H)$ on $\eb$ commutes with the projection onto the submodule $\crho(H)$, applied on both factors of $\eb$. We have already seen that $B$ is convex and has boundary $H$, thus by Proposition \ref{representations}, $1\otimes \Phi_\tim$ extends to a representation of $\crho(B)$ on $\fre$.

Note that as the representation of $\crho(B)$ on $\eb$ is faithful, the representation of $\B(\crho(B)\otimesh(\eb\oplus\ebp))$ on $\fre$ defined by $x\mapsto 1\otimes x$ is also faithful. Thus the algebra of adjointable operators on $\crho(B)\otimesh(\eb\oplus\ebp)$ is identified with a subalgebra of $\B(\fre)$. The image of $\crho(B)$ under the extension of $1\otimes \Phi_\tim$ lies in this subalgebra. In other words $\Phi_\tim$ also extends to give a representation of $\crho(B)$ on $\crho(B)\otimesh(\eb\oplus\ebp)$.

It remains to check that $(\Phi_\tim,\phi')$ defines a quasi-homomorphism for all $\tim$, and hence that the quasi-homomorphisms $(\phi'',\phi')$ and $(\phi,\phi')$ agree on $K$-theory. It is sufficient to check that $\Phi_\tim(x)-\phi(x)$ is a compact operator for the generators $x=T^B_g$ and $x=T^{B}_s$: given this, $\Phi_\tim(x)-\phi'(x)=(\Phi_\tim(x)-\phi(x))+(\phi(x)-\phi'(x))$ is compact as required. Since $v$ is a compact perturbation of the identity, the same holds for $v_\tim$ for all $\tim$, so $\Phi_\tim(T^B_g)$ agrees with $\phi(T^B_g)$ modulo compact operators, while $\Phi_\tim(T^{B}_s)=\phi(T^{B}_s)$ by definition. This completes the proof.
\end{proof}

\section{$HNN$-extensions} \label{HNN}

In this section we consider $K$-theory exact sequences for $HNN$-extensions. Let $G$ be a discrete group, $H$ a subgroup of $G$ and $\theta$ a monomorphism from $H$ to $G$. Let $K$ denote the image of $H$ in $G$. The $HNN$-extension $\Gamma=G\mathop{*}_H$, is the group generated by $G$ along with one extra generator $t$, and subject to the additional relations $tht^{-1}=\theta(h)$ for $h\in H$. In this context there is again an $H$-almost invariant subspace $B$ of the group $\Gamma$, and hence a corresponding exact sequence
$$0\to \crho(H)\otimes\K \to \crho(B)\to \crho(\Gamma)\to 0.$$
Specifically, we take the $H$-infinite, $H$-almost invariant subset $B=B_{H,G}$ of $\Gamma$,  as in Section \ref{sub:Bass-Serre}. We will prove that $K_*(\crho(B))\cong K_*(\crho(G))$, hence obtaining a 6-term exact sequence to compute the $K$-theory of $\crho(\Gamma)$ in terms of the $K$-theory groups for $\crho(H)$ and $\crho(G)$. The set $B$ will be right $G$-invariant, giving us a unital representation $\mu:\crho(G)\to \crho(B)$, and as in Section \ref{free-products} we will construct a quasi-homomorphism to invert this at the level of $K$-theory.

\medskip
As noted in Section \ref{sub:Bass-Serre}, the set $B$ can be described as those group elements which cannot be written as a reduced word beginning with $t^{-1}$. By construction the left-stabiliser of $B$ is $H$, while the left-stabiliser of $tB$ is $H^t=K$. The set $tB$ can be described as elements of $\Gamma$ which can be written as a reduced word starting with $t$.

We will build a quasihomomorphism from $\crho(B)$ to $\crho(G)$ using the bijection
\begin{equation}\label{HNN bijection}
(G\times_H\Gamma)\,\sqcup\, G \,\sqcup\, (G\times_K tB)  \leftrightarrow  (G\times_H B) \,\sqcup\, \Gamma
\end{equation}
To construct the bijections we decompose $\Gamma$ as follows. Take $\tree$ and remove all edges emanating from $eG$. This decomposes the tree into three subsets $\tree_L,\tree_R$ and the singleton $eG$, where $\tree_L$ is the union of the components joined to $eG$ by an edge labelled $t$, and $\tree_R$ is the union of the components connected to $eG$ by an edge labelled $t^{-1}$. The preimages of $\tree_L$ and $\tree_R$ under the map from $\Gamma$ to $\tree$ are denoted $\Gamma_L$ and $\Gamma_R$, while the preimage of the vertex $eG$ is simply $G$. We note that this gives us a decomposition of $\Gamma$ as the disjoint union of $\Gamma_L,\Gamma_R$ and $G$. All three subsets are both left- and right-$G$-invariant: right-invariance is automatic as they are preimages of subsets of $\tree$ while left-invariance follows from left-invariance of the subsets $\tree_L,\tree_R,G$ of the tree.

We observe that $\tree_L$ can alternatively be described as the product $G\calb^c$: recall that $\calb$ and $\calb^c$ are obtained by cutting the single edge $\epsilon$ from $t^{-1}G$ to $eG$, while $\tree_L$ is obtained by cutting all $G$-translates of this edge. Moreover the stabiliser of $\epsilon$ is $H$, hence we obtain a bijection $\tree_L\leftrightarrow G\times_H\calb^c$. This induces a corresponding bijection in the group $\Gamma_L\leftrightarrow G\times_H B^c$. Similarly $\tree_R\leftrightarrow G\times_Kt\calb$, and hence $\Gamma_R\leftrightarrow G\times_K tB$.

Using the two decompositions of $\Gamma$ as $\Gamma_L\sqcup G\sqcup \Gamma_R$ and $B\sqcup B^c$, the required bijections are now given by:
\begin{equation}\label{more HNN bijections}
\begin{matrix}
\bigl(\,G\times_H B &\sqcup&G\times_H B^c\,\bigr)&\sqcup&G&\sqcup&G\times_K tB\\
\updownarrow&&\updownarrow&&\updownarrow&&\updownarrow\\
G\times_H B &\sqcup& \bigl(\,\,\gammal\,\,\,&\sqcup &G&\sqcup &\quad\gammar\,\,\bigr)
\end{matrix}
\end{equation}

\medskip

As in Section \ref{free-products} the quasi-homomorphisms will be given by representations of $\crho(B)$ on Hilbert modules. We follow the same convention (Definition \ref{notation for Hilbert modules}) regarding notation for Hilbert modules.

The above HNN bijection (\ref{HNN bijection}) induces an isomorphism of $\crho(G)$-Hilbert modules
$$(\crho(G)\dispotimesh \ehgam)\oplus \crho(G)\oplus (\crho(G)\dispotimesk\etb)\xrightarrow[\cong]{U}(\crho(G)\dispotimesh\eb)\oplus\eggam.$$

We are now in a position to write down the quasi-homomorphism.  Given $x\in \crho(B)$ we form the operators $1\otimes x$ on $\crho(G)\otimesh\eb$ and $\pi(x)$ in $\eggam$. Now define
$$\psi(x)=(1\otimes x)\oplus \pi(x).$$

The automorphism of $\Gamma$ given by $\gamma\mapsto t\gamma t^{-1}$ induces an isomorphism from $\crho(H)$ to $\crho(H^t)=\crho(K)$. There is a corresponding map $V:\eb\to \etb$ given by $\sigma_b\mapsto \sigma_{tb}$, and this induces an isomorphism of pairs $(\crho(H),\eb)\cong(\crho(K),\etb)$. Conjugation by $V$ induces an isomorphism $\B(\eb)\cong \B(\etb)$; for $x\in \B(\eb)$ we denote the corresponding elements of $\B(\etb)$ by $\hat{x}$. We remark that for a generators $T^B_w$ of $\crho(B)$, the operator $\hat{T^B_w}$ is simply $T^{tB}_w$.

Now given $x\in \crho(B)$ we can form the operator $(1\otimes \pi(x))\oplus 0\oplus (1\otimes \hat{x})$ on
$$(\crho(G)\dispotimesh \ehgam)\oplus (\crho(G)\dispotimesk\etb)\oplus \crho(G).$$
Conjugating by $U$ we define
$$\psi'(x)=U^*\bigl((1\otimes \pi(x))\oplus 0\oplus(1\otimes \hat{x})\bigr)U.$$

As $\ehnn$ is a $\crho(G)$-module, the algebra of compact operators on this has the same $K$-theory as $\crho(G)$ by Morita equivalence. As usual, it is the $K$-theory of the compact operators on the Hilbert module, which is the target of the map on $K$-theory.

\begin{lemma}
\label{hnn-quasi}
The pair $\psi,\psi'$ define a quasi-homomorphism
$$\crho(B)\rightrightarrows \B(\dispehnn)\rhd \K(\dispehnn).$$
Moreover, the composition on $K$-theory
$$K_*(\crho(G))\xrightarrow{\mu_*}K_*(\crho(B))\xrightarrow{\psi_*-\psi'_*}K_*(\K(\dispehnn))
$$
is the canonical isomorphism.
\end{lemma}

\begin{proof}
The unitary $U$ is built from the bijections appearing in (\ref{more HNN bijections}), and we observe that each bijection is right-$G$ equivariant, and hence $U$ intertwines the right $\crho(G)$ actions on the Hilbert modules, cf.\ Remark \ref{right-equivariance}. It follows that the compositions $\psi\circ\mu,\psi'\circ\mu$ agree except on the $\crho(G)$-submodule of $\eggam$ generated by $\sigma_e$. This submodule is isomorphic to $\crho(G)$, and the restriction of $\psi\circ\mu$ to this submodule is simply the right-regular representation of $\crho(G)$, while $\psi'\circ\mu$ vanishes on this submodule. Hence the difference $\psi\circ\mu-\psi'\circ\mu$ takes $y\in\crho(G)$ to the rank-one operator $\xi\mapsto \ip{\xi,0\oplus\sigma_e}(0\oplus y\sigma_e)$.

We now consider $\psi(T^B_t)-\psi'(T^B_t)$. We have $\psi(T^B_t)=1\otimes T^B_t\oplus \rho(t)$, while $\psi'(T^B_t)=U^*((1\otimes \rho(t))\oplus 0\oplus (1\otimes T^{tB}_t))U$. First we compare the action of these on $\rho(g)\otimes \sigma_\gamma$ in $\crho(G)\otimesh\eb$. The operator $U^*$ includes this into $\crho(G)\otimesh\ehgam$, and thus we have
$$(\rho(g)\otimes \sigma_\gamma)\psi'(T^B_t)=(\rho(g)\otimes \sigma_\gamma)(1\otimes \rho(t))U=(\rho(g)\otimes \sigma_{\gamma t})U=\rho(g)\otimes \sigma_{\gamma t}$$
since $\gamma t\in B$ for $\gamma\in B$. This agrees with the action of $\psi(T^B_t)$. Now consider the action on $\sigma_\gamma\in \eggam$. The operator $U^*$ takes this to $\rho(\gamma)$ if $\gamma\in G$, and otherwise to $\rho(g)\otimes \sigma_{g^{-1}\gamma}$ where the factorisation of $\gamma$ is such that $g\in G$ and $g^{-1}\gamma\in B^c$ if $\gamma\in \Gamma_L$ and in $tB$ if $\gamma\in \Gamma_R$. In the case that $\gamma\in G$ we see that $\sigma_\gamma\psi'(T^B_t)=0$, and if $\gamma\in \Gamma_R$ then we get
$$\sigma_\gamma\psi'(T^B_t)=(\rho(g)\otimes \sigma_{g^{-1}\gamma})(1\otimes T^{tB}_t)U=(\rho(g)\otimes \sigma_{g^{-1}\gamma t})U=\sigma_{\gamma t}.$$
If $\gamma\in \Gamma_L$ then
$$\sigma_\gamma\psi'(T^B_t)=(\rho(g)\otimes \sigma_{g^{-1}\gamma})(1\otimes \rho(t))U=(\rho(g)\otimes \sigma_{g^{-1}\gamma t})U$$
which is $\sigma_{\gamma t}$ if $\gamma t\in \Gamma_L$, i.e.\ $\gamma t \notin G$, and is $\rho(\gamma t)\otimes \sigma_e\in \crho(B)\otimes \eb$ if $\gamma t \in G$.

On the other hand the action of $\psi(T^B_t)$ on $\sigma_\gamma$ is given by $\rho(t)$ so we simply have $\sigma_{\gamma t}$ for all $\gamma$. Thus we see that $\psi(T^B_t),\psi'(T^B_t)$ agree except on the submodule generated by $\sigma_e,\sigma_{t^{-1}}$. In particular the difference is compact.

Since $\crho(B)$ is generated by the image of $\mu$ along with $T^B_t$ we deduce that the pair $\psi,\psi'$ define a quasi-homomorphism, and the computation of $\psi\circ\mu-\psi'\circ\mu$ gives the required composition on $K$-theory.
\end{proof}

As in Section \ref{free-products} we will complete our computation by a homotopy argument, using Proposition \ref{representations}. To use this we must fix a presentation for $\Gamma$. We take the generating set $\Sigma$ to be $G\cup\{t,t^{-1}\}$ and we take the set of relations to be words of length at most $3$ in $G$ representing the identity along with words of the form $tht^{-1}k$ where $k=\theta(h)$, and their cyclic permutations.

As with the free product case, passing to a reduced word involves replacing a word with a single letter (e.g.\ replacing $tht^{-1}$ with $k$) and hence does not increase tracks. Thus any word that stays in $B$ can be simplified to a reduced word while staying in $B$. Two reduced words represent the same element if it is possible to get from one to the other by a sequence of moves which change a subword $g_itg_{i+1}$ to $g_i'tg_{i+1}'$ where $g_i=g_i'k$ and $g_{i+1}'=hg_{i+1}$ with $k=\theta(h)$, or similarly for $g_it^{-1}g_{i+1}$. These moves are three steps: $g_itg_{i+1}\rightsquigarrow g_i'ktg_{i+1}\rightsquigarrow g_i' t hg_{i+1}\rightsquigarrow g_i'tg_{i+1}'$. For a word staying in $B$, the first step stays in $B$ by right-$K$-invariance of $B$, the second by right-$H$-invariance of $B$, and the third is a reduction.

As $B$ is right-$G$-invariant and is translated into itself by $t$, the boundary of $B$ is those $b\in B$ such that $bt^{-1}\notin B$. Since $b\in B$ precisely when it cannot be written as a reduced word starting with $t^{-1}$, we have $bt^{-1}\notin B$ if and only if $b\in H$, so $bt^{-1}=t^{-1}k$ for $k=\theta(h)$. Hence the boundary of $B$ is the left-stabiliser $H$.

\begin{theorem}
The map $\mu_*:K_*(\crho(G))\to K_*(\crho(B))$ is an isomorphism, with inverse provided by $(\psi_*-\psi'_*)$.
\end{theorem}

\begin{proof}
We prove this by showing that the composition
$$K_*(\crho(B))\to K_*(\crho(G))\to K_*(\crho(B))$$
is the identity. Here we are identifying $K_*(\crho(G))$ with the $K$-theory of the algebra of compact operators on $\ehnn$, and similarly, $K_*(\crho(B))$ is identified with the $K$-theory of the compact operators on a $\crho(B)$-Hilbert module. Specifically, the composition is given by the quasi-homomorphism $\phi=1\otimes \psi,\phi'=1\otimes \psi'$ on
$$\crho(B)\dispotimesg (\dispehnn)\cong (\crho(B)\dispotimesh \eb)\oplus (\crho(B)\dispotimesg \eggam),$$
and the algebra of compacts on this module has the same $K$-theory as $\crho(B)$.

To distinguish elements of $\eb$ from elements of $\eggam$ we will denote the `basis vectors' by $\sigb_w\in \eb$ and $\siggam_w\in \eggam$. For $x\in \crho(B)$, we note that $\phi'(x)$ vanishes on the Hilbert module generated by $T^B_e\otimes \siggam_e$, and we define $\phi''(x)$ to be the sum of $\phi(x)$ and the rank-one operator $\xi\mapsto \ip{\xi,T^B_e\otimes\siggam_e}(x\otimes\siggam_e)$. Clearly $(\phi'',\phi')$ defines a quasi-homomorphism inducing the canonical isomorphism on $K$-theory. We will construct a homotopy from $(\phi,\phi')$ to $(\phi'',\phi')$, thus showing that the former also induces the canonical isomorphism.

Our computation of $\psi\circ\mu-\psi'\circ\mu$ shows that $\phi$ and $\phi''$ agree on the image of $\mu$. On the other hand we showed that $\psi(T^B_t)$ and $\psi'(T^B_t)$ agree except on the submodule generated by $\siggam_e,\siggam_{t^{-1}}$, i.e.\ the closed span of $\siggam_{g},\siggam_{gt^{-1}}$, for $g\in G$. Specifically, for $g\in G$ we have $\siggam_{gt^{-1}}\psi(T^B_t)=\siggam_{g}$ and $\siggam_{g}\psi(T^B_t)=\siggam_{gt}$, while $\siggam_{gt^{-1}}\psi'(T^B_t)=\rho(g)\otimes \sigb_e$ and $\siggam_{g}\psi'(T^B_t)=0$. Let $v$ denote the involution of $\ehnn$ which interchanges $\siggam_{g}$ and $\rho(g)\otimes \sigb_e$ while leaving the complementary submodule fixed. This is a compact perturbation of the identity, and we observe that $\psi(T^B_t)v$ agrees with $\psi'(T^B_t)$ except on the submodule generated by $\siggam_e$.

Tensoring with $1$ we see that $\phi(T^B_t)(1\otimes v)$ agrees with $\phi'(T^B_t)$ and hence with $\phi''(T^B_t)$ except on the submodule generated by $T^B_e\otimes \siggam_e$. We have $(T^B_g\otimes \siggam_e)\phi(T^B_t)(1\otimes v)=T^B_g\otimes \siggam_t$, while $(T^B_g\otimes \siggam_e)\phi''(T^B_t)=T^B_{gt}\otimes \siggam_e$. We now define $w$ to be the involution exchanging $T^B_g\otimes \siggam_t$ with $T^B_{gt}\otimes \siggam_e$, and again fixing the complement. We thus have $\phi(T^B_t)(1\otimes v)w=\phi''(T^B_t)$.

We can now define the homotopy on the generators $T^B_g$ for $g\in G$ and $T^B_t,T^B_{t^{-1}}$. We take $\Phi_\tau$ to be constant on $T^B_g$ setting thie equal to $\phi(T^B_g)=\phi''(T^B_g)$. For $T^B_t$ we define
$$\Phi_\tim(T^B_t)=\phi(T^B_t)(1\otimes v_\tim)w_\tim$$
where as usual $v_\tim,w_\tim$ denote the standard homotopies $v_\tim=\frac 12 ((1+e^{i\pi \tim})+(1-e^{i\pi \tim})v)$ and $w_\tim=\frac 12 ((1+e^{i\pi \tim})+(1-e^{i\pi \tim})w)$ for $\tim\in[0,1]$. We define $\Phi_\tim(T^B_{t^{-1}})=\Phi_\tim(T^B_t)^*$. As $v_\tim,w_\tim$ are unitaries $\Phi_\tim(T^B_t)\Phi_\tim(T^B_t)^*=\phi(T^B_t)\phi(T^B_t)^*=1$, while $\Phi_\tim(T^B_g)\Phi_\tim(T^B_g)^*=\phi(T^B_g)\phi(T^B_g)^*=1$ for any $g\in G$.

We observe that $1\otimes v, w$ commute with $\phi(T^B_h)$ for $h\in H$, so for $k=\theta(h)\in K$ we have
\begin{equation}\label{kt=th}
\begin{split}
\Phi_\tim(T^B_k)\Phi_\tim(T^B_t)&=\phi(T^B_k)\phi(T^B_t)(1\otimes v_\tim)w_\tim\\&=\phi(T^B_t)\phi(T^B_h)(1\otimes v_\tim)w_\tim\\
&=\phi(T^B_t)(1\otimes v_\tim)w_\tim\phi(T^B_h)=\Phi_\tim(T^B_t)\Phi_\tim(T^B_h).
\end{split}
\end{equation}

Hence $\Phi_\tim$ extends multiplicatively to products of partial translations in $\crho(B)$ corresponding to reduced words in $\Gamma$.

As in Theorem \ref{free product theorem} we pass to the tensor product with $\eb$, that is
$$\fre=(\eb\dispotimesh \eb)\oplus (\eb\dispotimesg \eggam).$$
Using Proposition \ref{representations}, we will show that $1\otimes \Phi_\tim$ extends to a representation of $\crho(B)$ on $\fre$ and hence $\Phi_\tim$ also extends to a representation of $\crho(B)$. Following the notation of the proposition, we will write $\alpha(\syl)$ for $1\otimes\Phi_\tim(T^B_\syl)$, and we extend this to words in the usual way.

\bigskip

Let $\gamma\in \Gamma$, and define $\EH\gamma$ to be the closed submodule of $\fre$ generated by $\sigb_b\otimes \sigb_w$ where $b,w\in B$ with $bw\in H\gamma$, and by $\sigb_b\otimes \siggam_w$ where $b\in B$, $w\in \Gamma$ with $bw\in H\gamma$. We remark that each elementary tensor of this form lies in precisely one such submodule; it follows that the spaces $\EH\gamma$ are mutually orthogonal and $\fre=\bigoplus\limits_{H\backslash\Gamma}\EH\gamma$.

We now establish the conditions of Proposition \ref{representations}.

(1) It is straightforward to check that $1\otimes\phi(T^B_\syl)$ takes $\EH\gamma$ to $\EH{\gamma\syl}$ for $\syl\in \Sigma$, and $1\otimes w,1\otimes v$ preserve each $\EH\gamma$. Hence $\alpha(\syl)$ takes $\EH\gamma$ to $\EH{\gamma\syl}$.

(2) If $w$ is a relation staying in $B$ then $w$ is either a word of length at most $3$ in $G$ or a word of the form $tht^{-1}k^{-1}$ or $k^{-1}tht^{-1}$ where $k=\theta(h)$. In the first case we note that $\Phi_\tim=\phi$ on $T^B_g$ $g\in G$, hence $\Phi_\tim$ satisfies the relation, and the same holds for $\alpha$. As $\Phi_\tim(T^B_k)\Phi_\tim(T^B_t)=\Phi_\tim(T^B_t)\Phi_\tim(T^B_h)$ and $\Phi_\tim(T^B_t)$ is an isometry we have
$$\Phi_\tim(T^B_k)=\Phi_\tim(T^B_t)\Phi_\tim(T^B_h)\Phi_\tim(T^B_t)^*$$
and pre/post-multiplying by $\Phi_\tim(T^B_k)^*$ yields the relations $\alpha(k^{-1}tht^{-1})=1$ and $\alpha(tht^{-1}k^{-1})=1$.

(3) If $w$ is a relation not staying in $B$ then $w$ is a word of the form $t^{-1}k^{-1}th$ or $ht^{-1}k^{-1}t$ where $k=\theta(h)$. Taking the adjoint of Equation \ref{kt=th} and tensoring with $1$ we have $\alpha(t^{-1}k^{-1})=\alpha(h^{-1}t^{-1})$, so using the fact that $\alpha(h)$ is unitary we get
$$\alpha(ht^{-1}k^{-1}t)=\alpha(t^{-1}t).$$
The product $\phi(T^B_{t^{-1}} T^B_{t})$ is the projection onto the complement of the submodule generated by $\sigb_e\otimes\sigb_e$, so $\alpha(t^{-1}t)$ is $1-p$ where $p$ is the projection onto the submodule generated by $(\sigb_e\otimes\sigb_e)(1\otimes v_\tim)(1\otimes w_\tim)$. Similarly tensoring Equation \ref{kt=th} with $1$ we have
$$\alpha(t^{-1}kth^{-1})=\alpha(t^{-1})\alpha(kt)\alpha(h^{-1})=\alpha(t^{-1})\alpha(th)\alpha(h^{-1})=\alpha(t^{-1}t)$$
so again the relation gives us $1-p$.

(4) Since $\alpha(ht^{-1}k^{-1}t)=1-p=\alpha(t^{-1}kth^{-1})$ for all $h\in H$ and $k=\theta(h)$, we have $\alpha(ht^{-1}k^{-1}t)=\alpha(t^{-1}k^{-1}th)$. Thus we see that $1-p$ is unchanged by conjugating by $\alpha(h)=1\otimes\phi(T^B_h)$ for $h\in H$. Taking $\hat\rho$ to be the restriction of $1\otimes \phi$ to $\EH{}$ we thus have that $\hat\rho$ commutes with $p$.

Thus by Proposition \ref{representations} it follows that $1\otimes \Phi_\tim$ extends to a representation of $\crho(B)$ on $\fre$. Faithfulness of the representation of $\crho(B)$ on $\eb$ ensures that this extension is given by tensoring an extension of $\Phi_\tim$ by $1$, and in particular that $\Phi_\tim$ extends to a representation of $\crho(B)$ on $(\crho(B)\otimesh \eb)\oplus (\crho(B)\otimesg \eggam)$.

It remains to check that $(\Phi_\tim,\phi')$ defines a quasi-homomorphism for all $\tim$, and hence that the quasi-homomorphisms $(\phi'',\phi')$ and $(\phi,\phi')$ agree on $K$-theory. It is sufficient to check that $\Phi_\tim(x)-\phi(x)$ is a compact operator for the generators $x=T^B_g$ and $x=T^{B}_t$ of $\crho(B)$: given this, $\Phi_\tim(x)-\phi'(x)=(\Phi_\tim(x)-\phi(x))+(\phi(x)-\phi'(x))$ is compact as required. Since $1\otimes v,w$ are compact perturbations of the identity, the same holds for $1\otimes v_\tim,w_\tim$ for all $\tim$, so $\Phi_\tim(T^B_t)$ agrees with $\phi(T^B_t)$ modulo compact operators, while for $g\in G$ we have $\Phi_\tim(T^{B}_g)=\phi(T^{B}_g)$ by definition. This completes the proof.
\end{proof}

\begin{appendices}
\section{Universal subspaces}
\label{universal}
Let $\Gamma$ be any infinite discrete group. We will say that a subspace $X$ of $\Gamma$ is \emph{universal} if for every $r$ and every subset $F$ of $B_\Gamma(e,r)$ there exists $x\in \Gamma$ such that $X\cap B_\Gamma(x,r)=xF$. Intuitively this means that for any subspace $B$ of $\Gamma$ we can find arbitrarily large pieces of $X$ that look like $B$.

As the condition holds for every subset of $B_\Gamma(e,r)$, in particular it holds for the ball, so a universal space is deep. It is easy to see that universal spaces exist. For example for $\Gamma=\Z$ we can construct a universal space by enumerating all binary sequences $0,1,00,01,10,11,000,001,\dots$, concatenating them as $0100011011\dots$ and taking the subset of $\N$ whose characteristic function is this binary sequence.

We remarked earlier that associated to a translation operator $T^X_{g_1}\dots T^X_{g_k}$ there is a track $(g,F)$ and two operators with the same track are equal.

In the universal case two operators of this form are equal if and only if they have the same track.

\begin{lemma}
If $X$ is universal then the operators $T^X_{g,F}$ are linearly independent. In particular if $T^X_{g,F}= T^X_{g',F'}$ then $F=F'$ and $g=g'$.
\end{lemma}
\begin{proof}
Let $T$ be a finite linear combination $T=\sum\limits_i \lambda_iT^X_{g_i,F_i}$ with non-zero coefficients. Then clearly for each $g\in \Gamma$ the sum $\sum\limits_{i,g_i=g} \lambda_iT^X_{g_i,F_i}$ is zero, so without loss of generality we may assume that $g_i=g$ for all $i$.

Choose an $i$ such that $F_i$ is minimal with respect to inclusion, and take $r$ sufficiently large that every $F_j$ lies in $B_\Gamma(e,r)$. By the universal hypothesis there is an $x\in \Gamma$ such that $X\cap B_\Gamma(x,r)=xF_i^{-1}=\{xh^{-1} \mid h\in F_i\}$. Then $T^X_{g,F_i}\delta_x=\delta_{xg^{-1}}$. As $F_i$ is minimal, for each $j\neq i$ there is some $h\in F_j\setminus F_i$, so $xh^{-1}$ is not in $X$, and $T^X_{g,F_j}\delta_x=0$. It follows that $T\delta_x=\lambda_i\delta_{xg^{-1}}$. Hence no such $T$ is zero, and we conclude that the operators $T^X_{g,F}$ are linearly independent.
\end{proof}

Now let $B$ be any subspace of $\Gamma$. We can define a map $\phi$ from the dense subalgebra of $\crho X$ spanned by operators $T^X_{g,F}$ to the dense subalgebra of $\crho B$ spanned by operators $T^B_{g,F}$ as follows. We define $\phi(T^X_{g,F})=T^B_{g,F}$; this is well defined as if $T^X_{g,F}= T^X_{g',F'}$ then $(g,F)=(g',F')$ so $T^B_{g,F}= T^B_{g',F'}$. Moreover as the operators $T^X_{g,F}$ are linearly independent we can extend by linearity to the whole dense subalgebra. It is easy to see that $\phi$ is a homomorphism of $*$-algebras.

The following proposition tells us that this extends to the completed algebras, hence establishing a universal property for the algebra $\crho X$.

\begin{proposition}
If $X$ is a universal subspace, and $B$ is any subspace of $\Gamma$, then there is a unique $*$-homomorphism $\phi:\crho X\to \crho B$ extending the map $T^X_{g,F}\mapsto T^B_{g,F}$.
\end{proposition}

\begin{proof}
Given the above remarks, it suffices to show that the map on dense subalgebras is contractive. Let $T^X$ be a finite sum $\sum_i \lambda_i T^X_{g_i,F_i}$ and let $T^B=\phi(T^X)=\sum_i \lambda_i T^B_{g_i,F_i}$. For any $\varepsilon>0$ there exists a finitely supported vector $\xi=\sum \xi_j \delta_{y_j}$ of norm $1$, such that $\norm{T^B\xi}>\norm{T^B}-\varepsilon$.

Take $r$ sufficiently large that $B_\Gamma(e,r)$ contains $y_jF_i^{-1}$ for all $i,j$. By the universal property there exists $x\in \Gamma$ such that $X\cap B_\Gamma(x,r)=x(B\cap B_\Gamma(e,r))$. It follows that for $h\in F_i$ we have $y_jh^{-1}\in B$ if and only if $xy_jh^{-1}\in X$. Thus for $\zeta=\sum \xi_j \delta_{xy_j}$, the vector $T^X\zeta$ is the left translation by $x$ of $T^B\xi$. In particular $\norm{T^X\zeta}=\norm{T^B\xi}> \norm{T^B}-\varepsilon$, hence $\norm{T^B}\leq \norm{T^X}$ as required.
\end{proof}

It follows immediately from this universal property that if $X,X'$ are two universal spaces then the correspondence $T^X_{g,F}\leftrightarrow T^{X'}_{g,F}$ yields an isomorphism $\crho X\cong \crho X'$.

We can use the idea of universal spaces to illustrate the importance of the co-separability hypothesis in our definition of relative almost invariance in Definition \ref{relative-H-almost-invariant}. 
\begin{example}\label{example_coseparability}
Let $\Gamma = F_2$ be the free group with generators $a, b$. Let $U$ be a universal subspace for $\Gamma$ chosen to be in the set of reduced words beginning with $b$. 
Now take $X= \langle a\rangle U$ and take $B$ be the subset of $X$ given by $B = \{a^nu \mid n\geq 0, \; u\in U\}$. Then $B$ is relatively deep in $X$, hence we have 
an exact sequence
\[
0\rightarrow I(P,B)\longrightarrow \crho (B) \longrightarrow \crho(X)\rightarrow 0. 
\]
Moreover, the pair $B$ and $X$ satisfies the first condition of Definition \ref{relative-H-almost-invariant} because $B$ is the intersection of $X$ with a half space of the Cayley graph of $F_2$. On the other hand, both and $B$ and $X$ are universal spaces for $\Gamma$. Hence the map $\crho(B) \rightarrow \crho (X)$ is an isomorphism so $I(P,B)$ is trivial. Thus Theorem \ref{ai-ideal} implies that $B$ is not co-separable. 
\end{example}

\end{appendices}

\bibliographystyle{plain}

\end{document}